\documentclass[a4paper,12pt]{article}
\usepackage[T1]{fontenc}
\usepackage[utf8]{inputenc}
\usepackage{amsmath,amsfonts,amssymb,amsthm,amscd,mathrsfs,dsfont}
\usepackage[french,english]{babel}
\usepackage{fancyhdr}
\usepackage{color}%per scrivere con diversi colori
\usepackage{graphicx}%per inserire grafici
\usepackage[a4paper,top=4.5cm,bottom=3.5cm]{geometry}
\usepackage[all]{xy}
\usepackage{hyperref}
\hypersetup{hidelinks}
\usepackage[autostyle,english=british]{csquotes}
\renewcommand{\epsilon}{\varepsilon}

\newcommand{\C}{\mathbb{C}}
\renewcommand{\H}{\mathbb{H}}
\newcommand{\D}{\mathbb{D}}
\newcommand{\R}{\mathbb{R}}
\newcommand{\Z}{\mathbb{Z}}
\newcommand{\N}{\mathbb{N}}
\renewcommand{\P}{\mathbb{P}}
\newcommand{\alfab}{\mathscr A^ \mathbb Z}

\DeclareMathOperator{\id}{Id}

\DeclareMathOperator{\aff}{Aff}

\theoremstyle{plain}
\newtheorem{theo}{Theorem}[section]
\newtheorem{lemma}[theo]{Lemma}
\newtheorem{prop}[theo]{Proposition}

\theoremstyle{definition}
\newtheorem{defin}[theo]{Definition}

\newtheorem{rk}[theo]{Remark}

\begin{document}
\begin{otherlanguage}{english}
\titlepage

\begin{otherlanguage*}{french}
\begin{flushleft}
ECOLE POLYTECHNIQUE\\
Master 1\\
PASQUINELLI Irene
\end{flushleft}

\vspace*{\fill}
\begin{minipage}{1\textwidth}
\textbf{
\begin{center}
{\large RAPPORT DE STAGE DE RECHERCHE}\\
\bigskip
{\huge \textit{Cutting sequences in \\ Veech surfaces}}\\
\noindent\makebox[\linewidth]{\rule{\textwidth}{.4pt}}\\
\mbox{}\\
\underline{\textit{NON CONFIDENTIEL}}\\
PUBLICATION 
\end{center}
}
\end{minipage}
\vfill

\begin{flushleft}
\underline{Option}: Département de MATHEMATIQUES\\
\underline{Champ de l'option}: Systèmes Dynamiques\\
\underline{Directeur de l'option}: Charles Favre\\
\underline{Directeur de stage}: Corinna Ulcigrai\\
\underline{Dates du stage}: 08 avril -- 31 ao\^ut 2013\\
\underline{Nom et adresse de l'organisme}:\\
University of Bristol\\
School of mathematics\\
University Walk\\
BS8 1TW\\
Bristol\\
United Kingdom
\end{flushleft}
\end{otherlanguage*}

\clearpage

\begin{abstract}

A cutting sequence is a symbolic coding of a linear trajectory on a translation surface corresponding to the sequence of sides hit in a polygonal representation of the surface.

We characterize cutting sequences in a regular hexagon with opposite sides identified by translations exploiting the same procedure used by Smillie and Ulcigrai in \cite{octagon} for the regular octagon.

In the case of the square, cutting sequences are the well known Sturmian sequences. We remark the differences between the procedure used in the case of the square and the one used in the cases of the regular hexagon and regular octagon. We also show how to adapt the latter to work also in the case of the square, giving a new characterization for this case.

We also show how to create a dictionary to pass from the symbolic coding with respect to the hexagon to the symbolic coding with respect to the parallelogram representing it  in the space of flat tori.

Finally we consider the Bouw-M\"oller surfaces $\mathscr M(3,4)$ and $\mathscr M(4,3)$ and we use their semi-regular polygons representations to prove our main result, which is a theorem analogous to the central step used to characterize cutting sequences in the previous cases.

\end{abstract}

\begin{otherlanguage}{french}
\begin{abstract}

Une suite de codage est un code symbolique d'une trajectoire linéaire dans une surface de translation et elle correspond à la succession des c\^otés croisés dans une représentation polygonale de la surface.

Nous caractérisons les suites de codage dans un hexagone régulier avec c\^otés parallèles identifiés par des translations en utilisant la m\^eme procédure que celle présentée par Smillie et Ulcigrai dans \cite{octagon}. 

Dans le cas du carré, les suites de codage sont les bien connues suites de Sturm. Nous montrons les différences entre la procédure utilisée dans le cas du carré et celle utilisée pour les cas de l'hexagone régulier et de l'octogone régulier. Nous montrons aussi comment adapter la deuxième pour l'appliquer au cas du carré aussi et nous donnons donc une nouvelle caractérisation dans ce cas.

Nous montrons aussi comme créer un dictionnaire pour passer du code symbolique lié à l'hexagone à celui lié au parallélogramme qui représente l'hexagone dans l'espace des tores plats.

Enfin nous considérons les surfaces de Bouw-M\"oller $\mathscr M(3,4)$ et $\mathscr M(4,3)$ et nous utilisons ses représentations en polygones semi-réguliers pour démontrer notre résultat principal, qui est un théorème analogue à l'étape centrale de la démonstration de la caractérisation des suites de codage dans les cas précédents.

\end{abstract}
\end{otherlanguage}

\clearpage

\tableofcontents

\section*{Introduction}

This work deals with cutting sequences of linear trajectories in
translation surfaces.

A translation surface, precisely defined in Paragraph
\ref{Veech}, is obtained by polygons, gluing pairs of parallel
sides by translations. In the representation of the surface by
polygons, linear trajectories are nothing else but straight lines
in the interior of the polygons. When the line hits one side, it
comes out from the corresponding point of the glued side.

We can code trajectories by labelling each pair of identified sides
with a letter and by recording the sequence of letters of sides
hit by the trajectory. Supposing a trajectory does not hit the
vertices, this gives us a bi-infinite sequence of letters in the
alphabet we are using. This is called the \emph{cutting sequence}
of the trajectory.

We want to try to understand which information about the
trajectory we can recover from the cutting sequence. For example,
can we characterize the cutting sequences in the space of
bi-infinite sequences? In other words, given a bi-infinite
sequence is it possible to find out whether it is a cutting
sequence of a certain trajectory or not? And if it is, can we
recover the direction of such trajectory?

The problem has already been studied in many cases. The first well
known and long studied one is a square with opposite sides
identified, treated for example by Caroline Series
in~\cite{square}. Since the square tessellates the plane, the
problem is equivalent to coding straight lines in the plane with
respect to a horizontal and vertical grid. The cutting sequences
obtained in this way are called \emph{Sturmian sequences} and
appear in various areas of mathematics.

The second case studied was the one of the regular octagon,
together with other regular $2n$-gons for $n \geq 4$, described by
John Smillie and Corinna Ulcigrai in \cite{octagon}. There, they
explain that in this case, characterizing cutting sequences is a
bit more complicated but still possible, in a similar way to
Sturmian sequences.

The third case is the one of regular pentagon and in general
regular odd-gons. In this case, since there are not anymore
opposite parallel sides to glue, Veech originally used two copies
of them glued together to build a translation surface. In
\cite{pentagon} Diana Davis analysed this case and showed that the
same procedure as Smillie and Ulcigrai can be applied for the
double regular pentagon and the general case of double regular
odd-gons.

The first part of my work, described in Section \ref{first}, was
to consider the case of the hexagon, ignored until now because by
gluing the opposite parallel sides we obtain a torus, and hence
the same kind of surface as in the case of the square. This case
turns out to be exactly as the case of the octagon and in the
first part of this work we explain the procedure used first in
\cite{octagon}, applying it to the case of the hexagon.

In the first part \ref{second} of the second section of this work,
we compare the methods used for the square and for the hexagon.
Even if the general method is the same, there is a slight
difference between definitions and tools used in the case of the
hexagon (and octagon and pentagon) and the one used in
\cite{square} for the square. We analyse the differences between
the two procedures and we show that the first method can be
adapted to work also for the square and give an other
characterization of Sturmian sequences.

Then, in the second part of the second section (see \ref{third})
since the hexagon and the square in the space of flat tori are not
the same surface, we explain the link between symbolic coding for
the hexagon and the symbolic coding for the correspondent
parallelogram in the space of flat tori.

In the third part of this work, described in Section \ref{fourth}
we analyse a particular surface of a new family of translation
surfaces, the Bouw-M\"oller surfaces, introduced by Irene Bouw
and Martin M\"oller in \cite{BM} and described by Pat Hooper in
\cite{Hooper} as semi-regular polygons glued together.

The problem of studying cutting sequences on these surfaces is
still open, although some first results are obtained by Diana
Davis in \cite{BMDavis}.

Bouw-M\"oller surfaces share an important property with surfaces
obtained by gluing regular polygons: they are all Veech surfaces. The
idea is that a Veech surface is a translation surface which has plenty of affine
symmetries. More precisely, this means that it is a translation
surface whose Veech group (defined in \ref{Veech}) is a lattice in
$SL(2,\R)$. In particular, this means that the quotient of the
space of deformations under its action is a hyperbolic surface
with finite area.

The general idea to characterize cutting sequences is to define a
combinatorial operation on words called \emph{derivation} and show
that cutting sequence are infinitely derivable. The proof is based
on the fact that we can cut and paste the original polygon to
obtain a new presentation as a sheared polygon and show that the
cutting sequence of the same trajectory with respect to the new
sides is exactly the derived sequence of the original sequence. By
sending the sheared polygon back in the original one with a shear,
we obtain a new trajectory. We can repeat the argument and find
out that the derived sequence of a cutting sequence is still a
cutting sequence and hence cutting sequences are infinitely derivable.

Finally, it is well known that the characterization of Sturmian
sequences is intimately connected with the geodesic flow on the
modular surface, as explained by Caroline Series in
\cite{squarehyp}. Similarly, in the case of the octagon and of the
pentagon, as well as in the case of the hexagon, the
characterization is connected with the Teichm\"uller geodesic flow
on the Teichm\"uller disk, as shown from Smillie and Ulcigrai in
\cite{octagonteich} for the octagon and here in Paragraph
\ref{teich} for the hexagon.

This and the brief outline of the central observation, show that
we deeply use the Veech group and hence the property of being a
Veech surface. This is why the case of Bouw-M\"oller surfaces can
be treated similarly to the previous ones.

\newpage

\section{Cutting sequences for the hexagon}\label{first}

\subsection{Main definitions}

Following \cite{octagon} we can provide the same construction for
the hexagon. Consider the hexagon $E$ centred in the origin of
the axis and with one edge parallel to the horizontal axis. We now
label the edges as follow: let's call $A$ the horizontal sides and
continue labelling $B$ and $C$ clockwise. By gluing the opposite
parallel pairs of sides by translations, we obtain a Veech
surface that we will call $S_E$, of genus 1, with two fake
singularities given by the vertices. Let us consider now a linear trajectory in
the surface. Obviously, it projects to a straight line in the
hexagon. For simplicity, we suppose to have just bi-infinite
trajectories, never hitting the fake singularities, which means
that our trajectories never hit vertices of the hexagon.

We can now \emph{code} the trajectory by writing down the letters
of the sides hit by the trajectory. If we consider then the set of
all labels as an alphabet $\mathscr A =\{A,B,C\}$, we have that
the sequence of labels is a bi-infinite word in the alphabet, that
we will call \emph{cutting sequence} and denote as $c(\tau)$ where
$\tau$ is the trajectory.

\subsubsection{Transition diagrams}

We try now to study the cutting sequences of such a surface.

As for the octagon, up to applying a rotation of angle $\pi$ we
can consider $\theta \in [0,\pi]$ because this leaves the labels
unchanged.

The symmetries of the hexagon suggests that we can divide
$[0,\pi]$ in the sectors
$\Sigma_i=[\frac{i\pi}{6},\frac{(i+1)\pi}{6}]$ for $i=0, \dots, 5$
and consider the element of $D_6$ which sends the sector
$\Sigma_i$ back in $\Sigma_0$. They are expressed by the following
matrices:

\begin{align*}
\nu_0&=\id, & \nu_1 &=\rho_{\frac{\pi}{6}} \circ r_1 \circ
\rho_{-\frac{\pi}{6}}=
\begin{pmatrix}
\frac{1}{2} & \frac{\sqrt{3}}{2} \\
\frac{\sqrt{3}}{2} & -\frac{1}{2}
\end{pmatrix},
\\
\nu_2&=\rho_{-\frac{\pi}{3}}=
\begin{pmatrix}
\frac{1}{2} & \frac{\sqrt{3}}{2} \\
-\frac{\sqrt{3}}{2} & \frac{1}{2}
\end{pmatrix},
&\nu_3&=\rho_{-\frac{\pi}{6}} \circ r_2 \circ
\rho_{\frac{\pi}{6}}=
\begin{pmatrix}
-\frac{1}{2} & \frac{\sqrt{3}}{2} \\
\frac{\sqrt{3}}{2} & \frac{1}{2}
\end{pmatrix},
\\
\nu_4&=\rho_{-\frac{2 \pi}{3}}=
\begin{pmatrix}
-\frac{1}{2} & \frac{\sqrt{3}}{2} \\
-\frac{\sqrt{3}}{2} & -\frac{1}{2}
\end{pmatrix},
&\nu_5&=r_2=
\begin{pmatrix}
-1 & 0 \\
0 & 1
\end{pmatrix}.
\end{align*}

Note that we call $\rho_\theta$ the counter-clockwise rotation of
angle $\theta$ and we call $r_1$ and $r_2$ the reflections through
respectively the horizontal axis and the vertical axis.

They act on the labels (and hence on the cutting sequences) by
\emph{permuting} the labels and we can write explicitly the
permutations that they induce:

\begin{align*}
&\pi_0=\id, &\pi_1=(AC), \\
&\pi_2=(ABC), &\pi_3=(BA), \\
&\pi_4=(ACB), &\pi_5=(BC).
\end{align*}

We want now to try to understand which words in the alphabet can
be cutting sequences and which ones cannot.

First of all, we check which transitions can occur in a cutting
sequence, i.e. which couple of letters can be found one after the
other in a cutting sequence.

What we do is to fix a sector, and to try to understand (and
represent in a graph) where else can a trajectory in that range of
directions go, if it just hit a certain side. As we saw, it is
enough to study transitions for trajectories with direction in
$\Sigma_0$ because the others will be obtained by permuting the
labels.

\begin{figure}[h]
\centering
\includegraphics[width=0.3\textwidth]{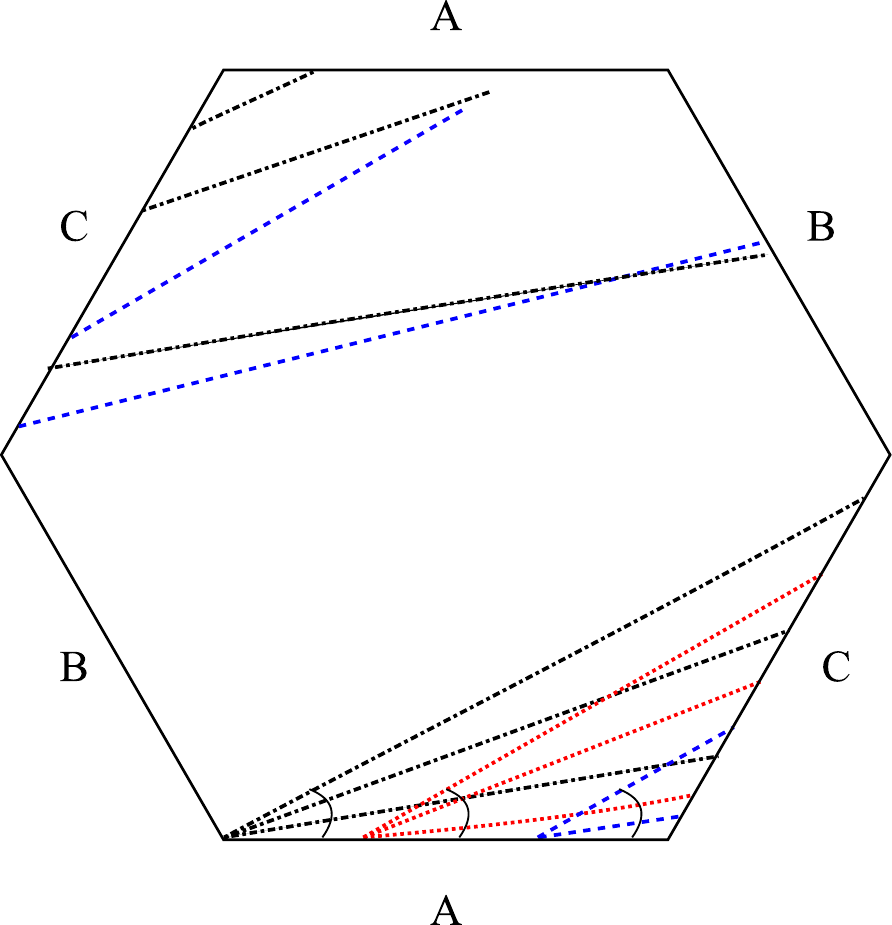}
\caption{Diagram $\mathscr{D}_0$} \label{D0}
\end{figure}

To describe them, we construct the related transition diagram
$\mathscr{D}_0$. A cutting sequence will then give a bi-infinite
path in the diagram.

We start from the bottom horizontal side labelled $A$ and we see in
Figure \ref{D0} that all possible trajectories in a direction in
$\Sigma_0$ necessarily go through $C$ and so on, obtaining the
following diagram.

\[
\mathscr D_0 \qquad \qquad  \xymatrix{ A \ar@/^1pc/[r] & C
\ar@/^1pc/[l] \ar@/^1pc/[r]
 & B \ar@/^1pc/[l] \ar@(ur,dr) }
\]

Similarly, we can construct all the diagrams for the other sectors
of directions by applying the permutations on the labels induced
by the action on the labels of the previous matrices.

The diagrams become hence:

\begin{align*}
\mathscr D_0 \qquad  &\xymatrix{ A \ar@/^1pc/[r] & C \ar@/^1pc/[l]
\ar@/^1pc/[r]
 & B \ar@/^1pc/[l] \ar@(ur,dr) }
&\mathscr D_1 \qquad  &\xymatrix{ C \ar@/^1pc/[r] & A
\ar@/^1pc/[l] \ar@/^1pc/[r]
 & B \ar@/^1pc/[l] \ar@(ur,dr) } \\ \\
\mathscr D_2 \qquad  &\xymatrix{ C \ar@/^1pc/[r] & B \ar@/^1pc/[l]
\ar@/^1pc/[r]
 & A \ar@/^1pc/[l] \ar@(ur,dr) }
&\mathscr D_3 \qquad  &\xymatrix{ B \ar@/^1pc/[r] & C
\ar@/^1pc/[l] \ar@/^1pc/[r]
 & A \ar@/^1pc/[l] \ar@(ur,dr) } \\ \\
\mathscr D_4 \qquad  &\xymatrix{ B \ar@/^1pc/[r] & A \ar@/^1pc/[l]
\ar@/^1pc/[r]
 & C \ar@/^1pc/[l] \ar@(ur,dr) }
&\mathscr D_5 \qquad  &\xymatrix{ A \ar@/^1pc/[r] & B
\ar@/^1pc/[l] \ar@/^1pc/[r]
 & C \ar@/^1pc/[l] \ar@(ur,dr) } \\
\end{align*}

\begin{rk}\label{commutation}
Since each matrix induces a permutation of labels as we just saw, we also have an action of each $\pi_i$ on the space of
bi-infinite words $\alfab$. Moreover, when passing to cutting
sequences, this action commutes with the action of the
corresponding matrices on the trajectory. In facts, for each $0
\leq k \leq 5$ we have
\[
c(\nu_k \tau)=\pi_k \cdot c(\tau).
\]
\end{rk}

By mapping a trajectory in direction $\theta \in [0,\pi]$ to $\Sigma_0$ with one of the $\nu_i$'s and up to permuting the labels which appear in the cutting sequence, we can assume that $\theta \in [0,\frac{\pi}{6}]$.

We will often need to consider a trajectory in some direction
$\theta \in [0,\pi]$ and \emph{send it back} to a trajectory with
direction in $\Sigma_0$ in order to simplify exposition and prove
our results using just our first sector. We will then often consider
the \emph{normal form} of a trajectory defined by

\begin{defin}\label{normaltraj}
The normal form of a trajectory in direction $\theta \in \Sigma_k$
is the trajectory
\[
n(\tau)=\nu_k \tau.
\]
\end{defin}

\subsubsection{Derivation}

Our aim is to give conditions under
which a word in the alphabet is or is not a cutting sequence.
First of all we can see that since we constructed transition
diagrams describing all possible transitions in a cutting
sequence, we have that all cutting sequences can be described as a
path in one of the diagrams $\mathscr D_i$.

To state this more precisely, we need to define that

\begin{defin}
A word in the alphabet is admissible if it describes a bi-infinite
path in one of the $\mathscr D_i$, i.e. if the transitions
appearing in the word are all allowed in one of the diagrams. In
that case we will say that the word is admissible in the diagram
$i$.
\end{defin}

Therefore, a necessary condition for cutting sequences is the following:

\begin{lemma}\label{csadmissible}
If $w \in \alfab$ is a cutting sequence, it is admissible in some
diagram $\mathscr D_k$.
\end{lemma}

In order to get a characterization, one needs a stronger condition
(see Proposition \ref{central} in the next section). In order to
define it, let us first introduce a combinatorial operation on
admissible words.

\begin{defin}
Given a letter $L$ in a word $w \in \alfab$, we say it is
sandwiched if the previous and the following letters are the same
letter $L'$.
\end{defin}

The operation on words is hence:

\begin{defin}
If $w \in \alfab$ is an admissible word, its derived sequence $w'
\in \alfab$ is the word obtained from $w$ keeping all sandwiched
letters and dropping the others.
\end{defin}

\begin{rk}
The condition of admissibility for $w$ is necessary for having a
bi-infinite word. In fact, one can easily check that starting from
every point of a diagram and considering all possible paths, after a finite
number of steps one continues to meet new sandwiched letters. Bi-infinite admissible sequences will then give
bi-infinite derived sequences.

For cutting sequences the admissibility condition is automatically satisfied.
\end{rk}

\begin{rk}\label{invariance}
It will also be useful to notice that the property of being
sandwiched for letters is invariant with respect to permutations.
Precisely:
\[
(\pi \cdot w)'=\pi \cdot w'
\]
where $\pi$ is a permutation.
\end{rk}

We also have a similar construction of a \emph{normal form} as in
\ref{normaltraj} for a certain subset of sequences. In fact we can
define

\begin{defin}\label{normalword}
If $w$ is a word admissible in just one diagram $\mathscr D_k$ for
a certain $k \in \{0,\dots,5\}$, the normal form of $w$ is the
word $n(w)=\pi_k \cdot w$.
\end{defin}

The concept of derivation on sequences allows us to describe an
important subset of sequences:

\begin{defin}
$w \in \alfab$ is derivable if it is admissible and its derived
sequence is still admissible.
\end{defin}

And recursively we have also

\begin{defin}
$w \in \alfab$ is infinitely derivable if it is derivable and
deriving $n$ times we obtain always a derivable sequence, for all
$n$.
\end{defin}

As we will state precisely and prove in the following section, all cutting sequences are
infinitely derivable. We will prove that the derived sequence of a
cutting sequence is again a cutting sequence.

\subsubsection{Veech group}\label{Veech}

What we still need in order to state and prove the condition is to
recall the definition of the affine automorphism group and of the Veech
group. Following the novel approach of \cite{octagon}, we will
allow orientation-reversing affine automorphisms.

We recall following \cite{transurf} that a translation surface is
a finite union (just one in this case) of polygons with
identifications of pairs of parallel sides such that

\begin{itemize}
\item the boundaries of the polygons are oriented so that the
polygon lies to the left;
\item Each side is glued to another parallel side of same length
by translation so that the two sides have opposite orientations.
\end{itemize}

Using this definition our surface $S_E$ is naturally a translation
surface.

Equivalently, a translation surface is a surface $S$ with a set of
singularities $\Sigma$ such that $S \setminus \Sigma$ is covered
by an atlas $\{U_\alpha,\phi_\alpha\}$ for which the changes of
charts on the intersections are translations:
\[
\phi_\alpha \phi_\beta^{-1}(v)=v+c.
\]

The singular points are cone singularities which possibly arises
from the vertices of the polygons (in our case they are just fake
singularities). The surface inherits a flat metric which allows us
to talk about trajectories, which are geodesics with respect to
the metric.

Using this second definition, we can state how to deform a
translation surface:

\begin{defin}
Given $\eta \in GL(2,\R)$ the deformation of a translation
surface $S$ by $\eta$ is the translation surface $S'$ obtained by
applying $\eta$ to the images under the charts, which means considering the atlas $\{U_\alpha, \eta \cdot \phi_\alpha\}$ and fixing the singularities.

If we use the first definition, we can consider the surface
obtained by gluing the images of the sides of the polygons after
applying $\eta$ to the polygons themselves.

A deformation of $S$ by $\eta$ is hence the natural application
\[
\Phi_\eta \colon S \to S'=\eta \cdot S
\]
and with abuse of notations we will sometimes refer to such a map by writing deformation of $S$.
\end{defin}

\begin{rk}
Obviously, it sends linear trajectories $\tau$ in $S$ to linear
trajectories $\eta \tau$ in $S'$ and the set of singularities in
the set of singularities.
\end{rk}

We just explained how to associate a matrix in $GL(2,\R)$ to a
deformation of $S$ which is a particular affine diffeomorphism:

\begin{defin}
An affine diffeomorphism $\Psi$ between $S$ and $S'$ translation
surfaces is a homeomorphism which preserve the set of
singularities i.e. sends $\Sigma$ in $\Sigma'$ andoutside  is a
diffeomorphism such that the derivative $D\Psi_p$ at the
point $p \in S$ between the tangent spaces does not depend on $p$.
\end{defin}

In fact the derivative of the deformation $\Phi_\eta$ is the
matrix $\eta$ in every point of $S \setminus \Sigma$.

Conversely, to each affine diffeomorphism, we can associate, by a map that we
call $V$, a
matrix in $GL(2,\R)$ as the derivative $D\Psi$ .

We will say that $S$ and $S'$ are affinely equivalent if it exists
an affine diffeomorphism between them, isometric if it exists such
affine diffeomorphism and its derivative is in $O(2)$ and translation
equivalent if there exists such an affine diffeomorphism $\Upsilon$ and its
derivative is the identity. In the last case, the map $\Upsilon$, up to cutting the
polygons composing $S$ in smaller pieces, can be defined locally
on each piece as a translation in order to give out the polygons
composing $S'$.

In the case $S=S'$ the affine diffeomorphisms are called affine
automorphisms  and form a group we will call $\aff(S)$.

Let us now consider $V$ that is called the \emph{Veech homomorphism}
for $S=S'$

\begin{align*}
V \colon \aff(S) &\to GL(2,\R) \\
\Psi &\mapsto D\Psi.
\end{align*}

Naturally the image is in $SL_\pm(2,\R)$, which is the subgroup of
matrices with determinant $\pm 1$ and we will call such image the
\emph{Veech group} of $S$ and write $V(S)$. Practically speaking
it is the set of derivatives of affine automorphisms of a surface.

Traditionally, the Veech group is considered to be the
intersection of such a group with $SL(2,\R)$ i.e. the
orientation-preserving elements that we will call $V^+(S)$.
Moreover, sometimes it is considered to be the projection of such
a group in $PSL(2,\R)$ which we will call $V_P^+(S)$ while the
projection in $PGL(2,\R)$ will be denoted by $V_P(S)$.

The Kernel of $V$ is the set of translation equivalences of $S$
and it is trivial if and only if given the derivative, the affine
automorphism is automatically determined because injectivity
guarantees that we can't have two preimages. In that case the most
classical example of affine automorphism is hence the unique one
and it can be described as follow: given $S$ as identified
polygons and its derivative $\eta$ we can compose the canonical
map described above $\Phi_\eta \colon S \to \eta \cdot S$ with a
translation equivalence $\Upsilon \colon \eta \cdot S \to S$
uniquely determined by the derivative if the kernel of $V$ is
trivial. We obtain an affine equivalence $\Psi_\eta \colon S \to
S$.

An other useful way to describe the Veech group of our surface
$S_E$ is to consider the map

\begin{align*}
h \colon SL(2,\R) &\to SL(2,\R) \cdot S_E \\
A &\mapsto A \cdot S_E.
\end{align*}

It's kernel is the set $V(S_E)=\{A \colon A \cdot S_E \doteq
S_E\}$ where "$\doteq$" means that the two surfaces can be
obtained one from the other by a cut and paste map which actually
is a translation equivalence. This means that each time we apply
an element of the Veech group to the hexagon we will always have a
cut and paste map which sends the new polygon back to the previous
one and hence this shows the equivalence with the previous
definition.

Veech proved in \cite{veechlattice} that the
orientation-preserving Veech group of each translation surface
obtained by gluing regular polygons is a lattice in $SL(2,\R)$
which naturally means that our Veech group $V(S)$ is a lattice in
$SL_\pm(2,\R)$. Surfaces satisfying such property are called Veech surfaces.

But we can say something more about the structure of the Veech
group of our surface. First of all we have that each element of
the dihedral group $\eta \in D_6$ determines an affine
automorphism as explained before $\Psi_\eta \colon S_E \to S_E$
whose derivative is exactly $\eta$ which means that the dihedral
group is a subgroup of the Veech group. In particular, we have the
reflection with respect to the horizontal axis and the reflection
with respect to the straight line forming angle $\frac{\pi}{6}$
whith the horizontal axis:
\[
\alpha= \begin{pmatrix}
1 & 0 \\
0 & -1
\end{pmatrix}
\qquad \qquad \beta= \begin{pmatrix}
\frac{1}{2} & \frac{\sqrt 3}{2} \\
\frac{\sqrt 3}{2} & -\frac{1}{2}
\end{pmatrix}
\]
and their counterpart as affine automorphisms $\Psi_\alpha$ and
$\Psi_\beta$.

By adapting the proofs of the orientation-preserving case exactly
as for the octagon in \cite{octagon} we can describe precisely the Veech group for the hexagon and state the following:

\begin{lemma}
The kernel of $V$ from $\aff(S_E)$ to $GL(2,\R)$ is trivial.

Moreover, the affine automorphisms group $\aff(S_E)$ is generated by
$\Psi_\alpha$, $\Psi_\beta$ and $\Psi_\gamma$ and hence the Veech
group is generated by their derivatives $\alpha$, $\beta$,
$\gamma$, where $\gamma = \begin{pmatrix}
-1 & 2 \sqrt 3 \\
0 & 1
\end{pmatrix}$, the Veech shear described in the next
paragraph.
\end{lemma}

\subsection{Necessary condition}

\subsubsection{Veech element}

We will now describe the third generator of the Veech
group $\gamma$.

We know that if we have a decomposition in cylinders with
commensurable moduli, then we can define an affine automorphism
with a shear as derivative which acts as a multiple of a Dehn
twist on each cylinder.

The best thing is to show it on an example. Let us  consider a
rectangle with two parallel sides identified, say the vertical
sides. Its modulus, which we call $m$, is the ratio $\frac
{height}{width}$ and is the slope of the diagonal. If
$\mu=\frac{1}{m}$ we now apply to the rectangular representation
the shear $\begin{pmatrix}
1 & \mu \\
0 & 1
\end{pmatrix}$. The horizontal lines remain horizontal, while the
vertical ones are sent in lines of slope $\frac{1}{\mu}=m$, which
means that the new sides are the horizontal side and the diagonal.
It is clear that the diagonal is somehow the geodesic which goes along
both the previous ones and hence we can describe the related
affine automorphism $\psi$ (with derivative the shear) as in Figure \ref{Dehntwist}, by "twisting" the two extremities of the
cylinder.

\begin{figure}[h]
\centering
\includegraphics[width=0.7\textwidth]{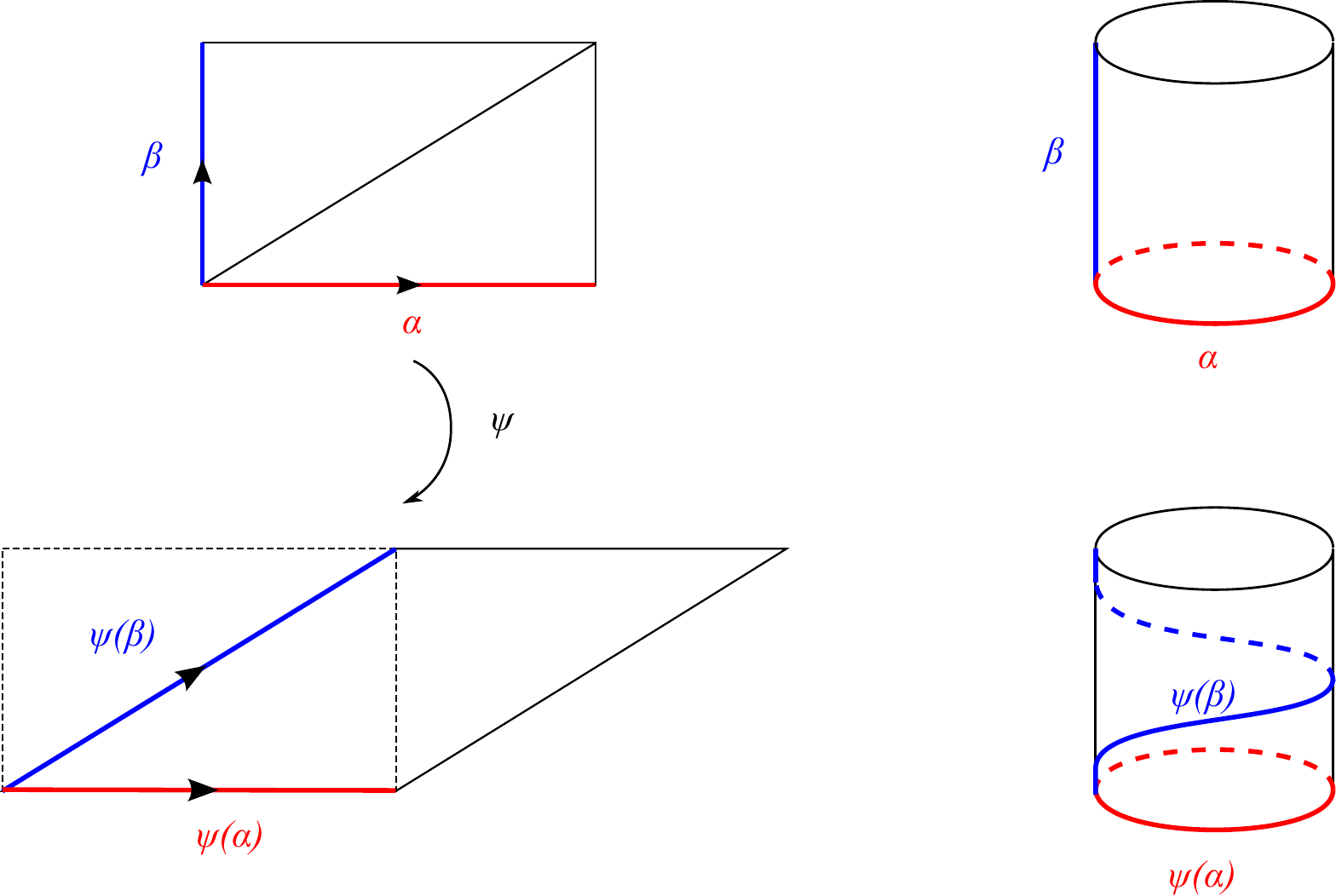}
\caption{Dehn twist on a  cylinder} \label{Dehntwist}
\end{figure}

When we have more than one cylinder, on each
cylinder we know how to do a Dehn twist and obviously a multiple
of it. Suppose hence that $S= \cup_i C_i$ and on each $C_i$ we
know how to apply $\psi_i^n$, for $\psi_i$ basic Dehn twist and $n
\in \N$. We define the affine automorphism
$\psi(x)=\psi_i^{n_i}(x)$ if $x \in C_i$ and $n_i$ is determined
by the condition of being an affine automorphism. In fact it is
equivalent to say that $D\psi_i=A$, a certain matrix independent
on $i$ and since $D(\psi_i^{n_i})=\begin{pmatrix}
1 & \mu_i^{n_i} \\
0 & 1
\end{pmatrix}$, we just need to choose a $\mu$ such that
$\mu=\mu_i^{n_i}=\mu_j^{n_j}$ for all $i,j$ (and it is possible if
the moduli are commensurable) and define $\psi$ such that its
derivative is

\[
A=\begin{pmatrix}
1 & \mu \\
0 & 1
\end{pmatrix}.
\]

It is independent on $i$ and it is a multiple of the standard Dehn
twist on each cylinder. Naturally, the matrix $A$ being the derivative
of an affine automorphism is an element of the Veech group.

In the case of the hexagon the corresponding orientation reversing
element (i.e. the composition of such an element with the
reflection with respect to the vertical axis) is exactly the
element $\gamma$ of the previous section, one of the generators of
the Veech group.

In order to find out explicitly that element, first of all we need
a decomposition in cylinders of the hexagon.

\begin{figure}[h]
\centering
\includegraphics[width=1\textwidth]{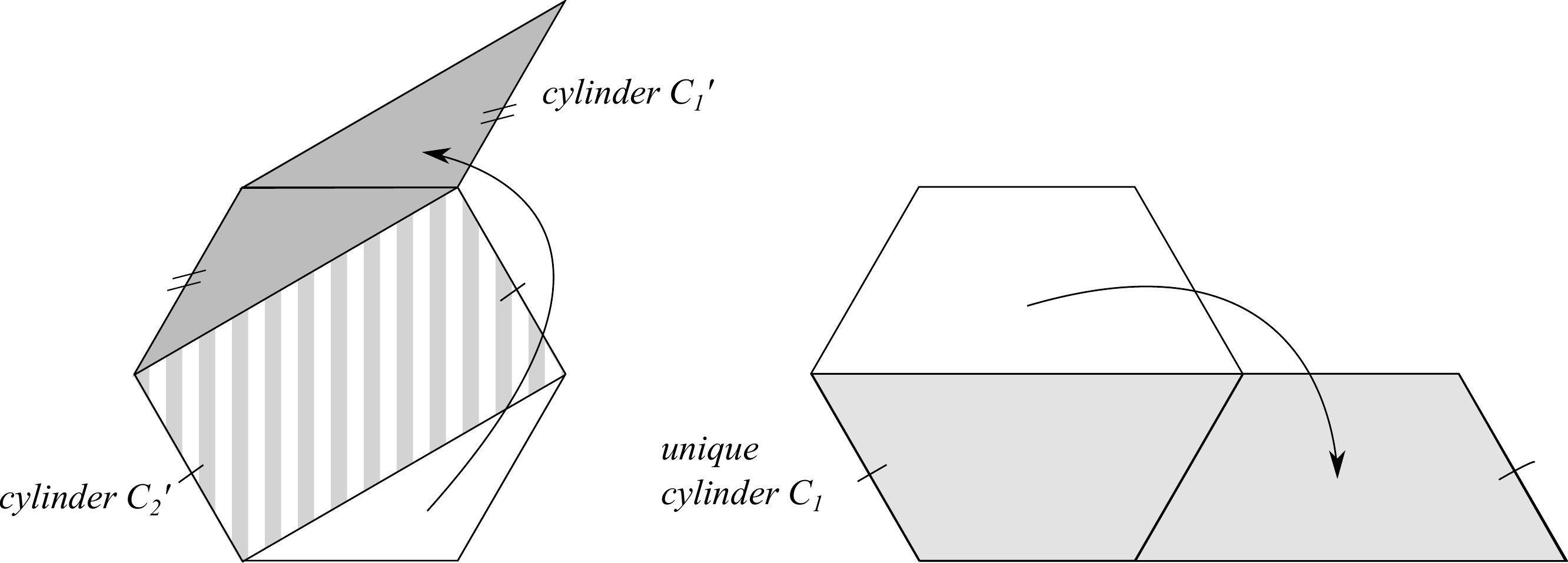}
\caption{cylinder decompositions} \label{C1-2}
\end{figure}

By looking at Figure \ref{C1-2}, we find out that we have two
possible cylinder decompositions. The two decompositions are in
the direction $\theta=0$ and $\theta=\frac{\pi}{6}$ and we
consider the horizontal one. It consists in cutting the hexagon
horizontally, and glueing the upper part to the side $C$, so that we
obtain a cylinder which has height $\frac{\sqrt 3}{2}$ and width
$3$. This means that its inverse modulus is $\mu=2 \sqrt{3}$.
Obviously we don't either have the problem of commensurability and
this gives us straightforwardly an affine diffeomorphism
$\Psi_\sigma$ which acts on the cylinder by a Dehn twist fixing
horizontal lines, sending vertical lines in lines of slope
$\frac{1}{\mu}$ and fixing singularities. Automatically, this
gives the element of the Veech group:
\[
\sigma=
\begin{pmatrix}
1 & 2\sqrt 3 \\
0 & 1
\end{pmatrix},
\]
as derivative of the affine diffeomorphism. The image of such an
element is a new (not regular) hexagon $E'=\sigma E$ shown in
Figure \ref{shear} and by definition of Veech group can be cut and
pasted again in the original hexagon $E$. We call such a cut and
paste map, a piecewise translation, by $\Upsilon_E$. His
derivative is obviously the identity and hence the affine
automorphism $\Psi_\sigma$ of the glued surface $S_E$ named before
is exactly $\Upsilon_E \circ \sigma$.

\begin{figure}[ht]
\centering
\includegraphics[width=1\textwidth]{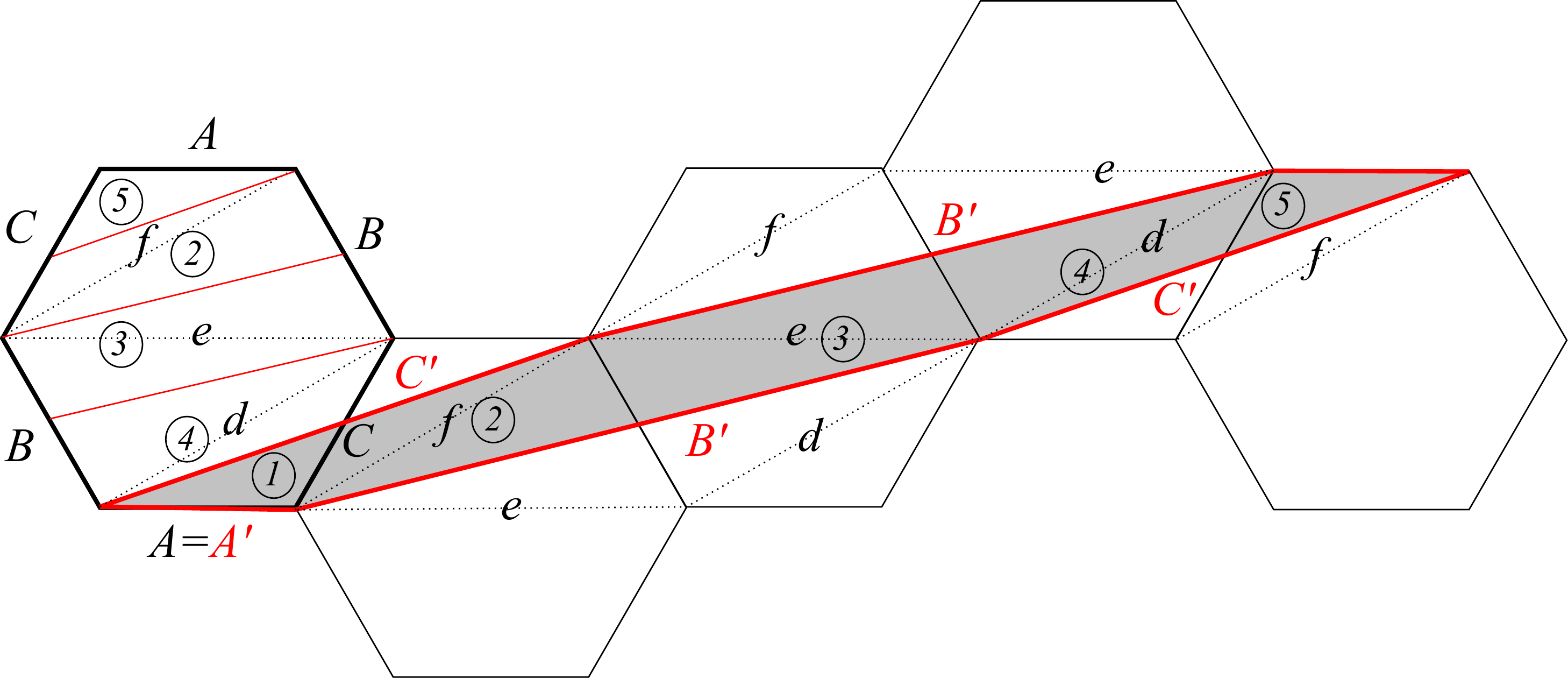}
\caption{Action of the shear} \label{shear}
\end{figure}

In order to simplify the notions that we will see later
on, we consider an other element of the Veech group, which acts in
the same way as $\sigma$ but reversing the orientation of the
labels. We define then
\begin{align*}
\gamma=\begin{pmatrix} -1 & 2\sqrt 3 \\
0 & 1
\end{pmatrix}
=\sigma \circ r_2
\end{align*}
which means that acting by $\gamma$ means reflecting the hexagon
with respect to the vertical axis (i.e. exchanging the labels $B$
and $C$, or equivalently reverting the orientation of the labels)
and then acting by the shear $\sigma$. It is easy to see that
$\gamma$ is an involution because its square is the identity.

\subsubsection{Statement}

The whole discussion about finding a necessary condition for a
sequence in the alphabet for being a cutting sequence is based on a
simple proposition, whose proof is purely combinatorial, but has a
huge geometrical hidden sense. The condition is

\begin{theo}
A cutting sequence is infinitely derivable.
\end{theo}

The proof is very simple using the following:

\begin{prop}\label{derivedcs}
The derived sequence of a cutting sequence is again a cutting
sequence.
\end{prop}

In fact, if the derived sequence is a cutting sequence, it is
admissible in some of the diagrams $\mathscr D_i$ and hence by
definition the original one is derivable. By repeating this
argument, we obtain the infinite derivability.

Our next step will be then to rephrase Proposition
\ref{derivedcs} in a more precise way and prove it.

\begin{prop}[Central Proposition]\label{central}
Given a trajectory $\tau$ in direction $\theta \in \Sigma_0$ with
cutting sequence $c(\tau)$, let us consider the trajectory
$\tau'=\Psi_\gamma \tau$. We have that the cutting sequence of the
new trajectory is exactly the derived sequence of the cutting
sequence of the old one. It means that
\[
c(\tau')=c(\tau)'.
\]
\end{prop}

Proving the second proposition, \ref{central}, is enough to obtain
the previous one, \ref{derivedcs}.

In fact suppose we have a trajectory $\tau$ in direction $\theta
\in \Sigma_i$ and call $w=c(\tau)$ its cutting sequence. Using the
second proposition we will show that the derived sequence $w'$ is
again a cutting sequence of a certain trajectory that will be
explicit later. As one immediately sees the problem is in showing
that we can somehow pass from trajectories in whatever direction
to trajectories in $\Sigma_0$ and a similar statement for a
different $\tau'$ is true. In fact by considering the normal form
defined by $n(w)=\nu_i\tau \in \Sigma_0$, we can apply Proposition \ref{central} and obtain that
\[
c(\Psi_\gamma \nu_i\tau)=c(\nu_i \tau)'=(\pi_i \cdot
c(\tau))'=\pi_i \cdot (c(\tau))'
\]
where the last inequalities are respectively Remark
\ref{commutation} and the invariance of the sandwich property
under relabelling. But then we can consider the trajectory
$\nu_i^{-1} \Psi_\gamma \nu_i \tau$ and a similar argument gives
\begin{align*}
c(\nu_i^{-1} \Psi_\gamma \nu_i \tau)&=\pi_i^{-1} \cdot c(
\Psi_\gamma \nu_i \tau)=\text{...previous equalites...}= \\
&=\pi_i^{-1}
\cdot \pi_i \cdot (c(\tau))'=c(\tau)'=w'.
\end{align*}

This means then that the derived sequence of the cutting sequence
of a trajectory $\tau$ in a direction in the $i-th$ sector is
again a cutting sequence and precisely it is the cutting sequence
of the trajectory $\nu_i^{-1} \Psi_\gamma \nu_i \tau$.

\subsubsection{Proof}

It remains now to prove Proposition \ref{central} and we do it
immediately.

\begin{proof}

We divide the proof in four simpler steps.

\begin{figure}[h]
\centering

\includegraphics[width=0.25\textwidth]{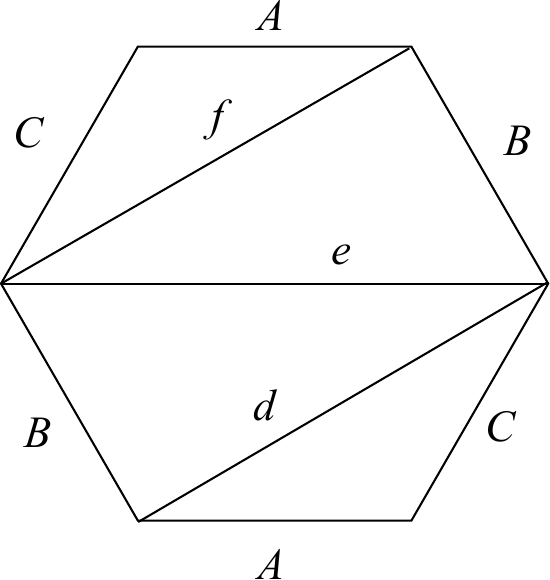}
\label{diagonal}
\end{figure}

\emph{First step}: First of all, we mark in the hexagon the
diagonals in directions $0$ and $\frac{\pi}{6}$ (the auxiliary
edges), and label them $d,e,f$ starting from the bottom, as in
Figure \ref{diagonal}. If now we consider the cutting sequence of
$\tau$, we can record whenever it crosses these diagonals by
writing down what we will call the \emph{augmented cutting
sequence} $\tilde c(\tau) \in \{A,B,C,d,e,f\}^\Z$. If we remember
that the cutting sequence of a trajectory in direction $\theta \in
\Sigma_0$ represents a path in the diagram $\mathscr D_0$, we can
construct an other diagram $\tilde {\mathscr D_0}$ on which the
same path will give us the augmented cutting sequence reading
both the pointed letters and the labels of the arrows. In fact,
we can see that for each transition, we can either cross or not an
auxiliary edge and it is always well determined. The new diagram
becomes hence:
\[
\tilde {\mathscr D_0} \qquad \qquad  \xymatrix{ A \ar@/^1pc/[r] &
C \ar@/^1pc/[l] \ar@/^1pc/[r]^f
 & B \ar@/^1pc/[l]^d \ar@(ur,dr)^e }
\]
and it will allow us to determine univocally the augmented
cutting sequence using $c(\tau)$.

\emph{Second step}: We now try to interchange the roles of edges
and vertices of the new diagram except for $A$, which is somehow
both a normal side and an auxiliary side, since it is a "diagonal"
in direction $0$. What we mean geometrically is to obtain a code
for the trajectory not based on the original sides, but on the
auxiliary sides, recording as before the crossings with normal
sides. $B$ and $C$ become hence labels of arrows and the diagrams
is

\[
\tilde {\mathscr D_0'} \qquad \qquad  \xymatrix{ && f
\ar@/^1pc/[drr]^B \ar@/^1pc/[dd]^B && \\
A \ar@(ul,dl)_C \ar@/^1pc/[urr]^C &&& & e \ar@/^1pc/[dll]^B \ar@(ur,dr)^B \\
&& d \ar@/^1pc/[ull]^C \ar@/^1pc/[uu]^C&&}
\]

The new diagram can determine a new "cutting sequence"
representing a path in $\tilde {\mathscr D_0'}$ and coding the
trajectory with respect to the auxiliary sides. We will call it
$\hat c(\tau) \in \{A,d,e,f\}^\Z$ and is obtained from $\tilde c(\tau)$
by erasing $B$'s and $C$'s.

\emph{Third step}: Let us now apply the cut and paste map
$\Upsilon^{-1}_E=E'$. As we saw, we obtain a new hexagon and we
label its sides with the same letters as the previous hexagon, but
primed (see Figure \ref{shear}). Obviously, our trajectory $\tau$
in $E$ is transported in $E'$ by the cut and paste and we can code
the trajectory with respect to the sides of the new hexagon. We
will call this new cutting sequence $\overline c(\tau) \in
\{A',B',C'\}^\Z$. The key remark is that we can write down $\overline
c(\tau)$ if we know $\hat c(\tau)$. How? It is evident that if we
have a trajectory in direction in $\Sigma_0$, and a parallelogram
with sides in directions $0$ and $\frac{\pi}{6}$, knowing the
cutting sequence with respect to the sides of the parallelogram
gives automatically the intersections with the positive slope
diagonal. Now, we can inscribe all the sides of $E'$ in
parallelograms made up of auxiliary edges (see Figure \ref{shear}). This means that if we
have the cutting sequence with respect to the auxiliary edges, we
can "augment" the sequence with diagonals and erase the useless
ones to obtain the cutting sequence with respect to the new sides.
It worth to draw the relative diagram $\mathscr D_0'$, noticing
that $A=A'$:
\[
\mathscr D_0' \qquad \qquad  \xymatrix{ && f
\ar@/^1pc/[drr] \ar@/^1pc/[dd]^{B'} && \\
A \ar@(ul,dl)_{C'} \ar@/^1pc/[urr] &&& & e \ar@/^1pc/[dll] \ar@(ur,dr)^{B'} \\
&& d \ar@/^1pc/[ull] \ar@/^1pc/[uu]^{C'} &&}
\]
which records when, in a transition between two auxiliary sides
(i.e. two sides of a parallelogram) we cross a primed side (i.e. a
diagonal). This means that if we know the path which determines
the sequence $\hat c(\tau)$, by reading the right letters ($A=A'$
and the labels of the arrows), we obtain $\overline c(\tau)$.

\emph{Fourth step}: Up to changing the labels $A=A'$, $B
\longleftrightarrow B'$ and $C \longleftrightarrow C'$ we have
\[
\overline c(\tau)=c(\tau)'.
\]
In fact, $c(\tau)$ can be read on the diagram $\tilde{\mathscr
D_0}$ by reading the label $A$ and the labels of the arrows.
Meanwhile, $\overline c(\tau)$ can be read on the diagram $\mathscr
D_0'$ in the same way. This means that it is enough to compare the
two diagrams $\tilde{\mathscr D_0}$ and $\mathscr D_0'$, but it is
easy to check that up to changing the letters as we said before,
$\mathscr D_0'$ is obtained from $\tilde{\mathscr D_0}$ erasing
from the arrow's labels exactly those letters which are not
sandwiched.

\emph{Fifth step}: Finally, by considering
\begin{align*}
\gamma \colon E' &\to E \\
B' &\mapsto B \\
C' &\mapsto C \\
\Upsilon ^{-1}_E \tau &\mapsto \tau'=\gamma \Upsilon^{-1}_E
\tau=\Upsilon^{-1}_E \gamma \tau=\Psi_\gamma \tau
\end{align*}
we exchange the labels as we needed in the previous step and we
send the trajectory $\tau$ seen as $\Upsilon ^{-1}_E \tau$ in $E'$
after the cut and paste in the new trajectory $\tau'$ of the
statement which will have cutting sequence with respect to $A=A'$,
$B=\gamma B'$ and $C=\gamma C'$ the sequence $c(\tau')=\overline
c(\tau)=c(\tau)'$ for the previous step and putting all together
we have exactly as in the statement of the central proposition
\[
c(\tau)'=c(\tau').
\]\qedhere
\end{proof}

We wish to point out that the definition of $\gamma$ as the
orientation reversing element is useful in order to have in the
end of the proof the same letters to substitute $C'$ with $C$ and
$B'$ with $B$. Obviously, nothing would change in the proof mixing
the two labels, which are only symbols to identify edges, but
using $\gamma$ instead of $\sigma$ allows us a simpler exposition.

\subsection{Farey map}

\subsubsection{Definition}

As in \cite{square} for the square and in \cite{octagon} for the
$2n$-gons, we want now to describe a way of producing a continued
fraction expansion for the directions. Exactly in the same way and
with the same proof as in \cite{octagon} we will describe it as the
itinerary of a map and use it to recover the direction of a
trajectory given its cutting sequence.

We want to find out a way of acting with matrices on the space of
directions $\R\P^1$. Each sector $\Sigma_i$ corresponds to a
sector $\tilde \Sigma_i \in \R\P^1$ and each $\nu_i$ induces hence
an action $\nu_i \colon \tilde \Sigma_i \to \tilde \Sigma_0$.

We now consider the angle coordinate in $\R\P^1$ given by $\theta
\in [0,\pi]$ and define the \emph {Farey map} as the piecewise map
defined on each $\tilde \Sigma_i$ by the branch $F_i$ given by the
action of $\gamma \nu_i$.

\begin{defin}
The Farey map $F \colon \R\P^1 \to \R\P^1$ is defined in the angle
coordinate $\theta$, for $i=0,\dots,5$ by
\[
F_i=\cot^{-1}\left( \frac{a \cot \theta +b}{c \cot \theta
+d}\right), \quad \text{ if } \theta \in \Sigma_i \text{ and }
\begin{pmatrix}
a & b \\
c & d
\end{pmatrix}=\gamma\nu_i.
\]
\end{defin}

In Figure \ref{fareygraph} is the graph of such a map in the
coordinates.

\begin{figure}[h]
\centering
\includegraphics[width=0.5\textwidth]{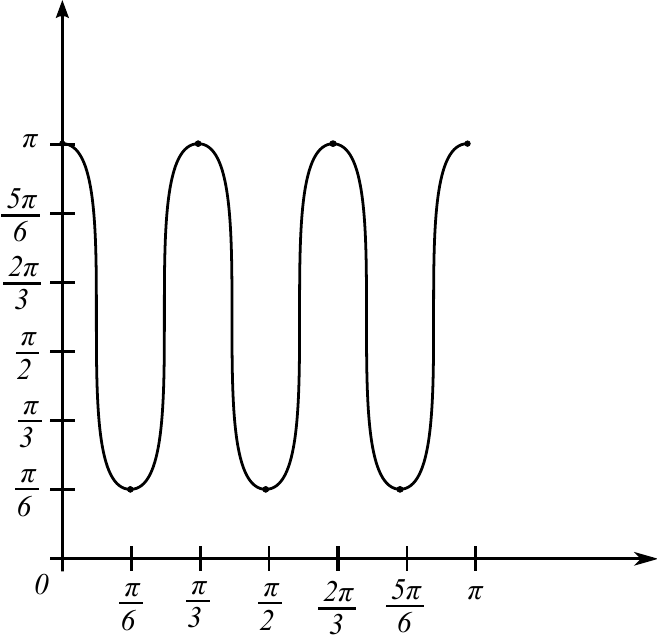}
\caption{Graph of the Farey map} \label{fareygraph}
\end{figure}

\begin{rk}\label{contimage}
As the graph easily shows, it is a continuous map and since each
branch is monotonic, they have an inverse map well defined.
Moreover, the image of the Farey map is always contained in
$[\frac{\pi}{6},\pi]$.
\end{rk}

\begin{rk}
The utility of the Farey map goes much further than we can explain
here. Still, we can't miss to point out an important feature
dealing with cutting sequences.

In fact, by the definition of the first branch $F_0$, since
$\nu_0$ is the identity, we can find out something more about the
trajectory $\tau'$ whose cutting sequence is the derived sequence
of $\tau$, as described in \ref{central}. In fact if $\tau$ is a
trajectory in direction $\theta \in \Sigma_0$ as in the
proposition, the new trajectory $\tau'=\Psi_\gamma \tau$ is in
direction $F(\theta)$.
\end{rk}

\subsubsection{Hexagon Farey expansion}

Starting from a $\theta \in [0,\pi]$ we can code its itinerary by
applying repeatedly the right branch of $F$. In order to do that,
we assign to the angle $\theta$ the sequence $s_0,s_1,\dots$ such
that $F^i(\theta) \in \Sigma_{s_i}$.

The previous remark, \ref{contimage}, shows us that such an
itinerary is always such that $s_0 \in \{0, \dots, 5\}$, while
$s_i \in \{1, \dots, 5\}$ for $i \geq 1$. We will call $S^*$ such
a subset of the set of sequences $\{0,\dots,5\}^\N$.

As the itinerary of the Gauss map for the square gives the
continued fraction expansion of the direction, in that case too
the itinerary determines uniquely the direction $\theta$ at the
beginning. In fact if we write
\[
\Sigma[s_0;s_1, \dots, s_k]= F^{-1}_{s_0}F^{-1}_{s_1} \dots
F^{-1}_{s_k}[0,\pi]
\]
for the set of all direction whose itinerary starts with the
segment $s_0, s_1, \dots s_k$, with the same proof as in
\cite{octagon} for a given infinite itinerary $\{s_k\} \in S^*$ we
can consider the intersection
\[
\bigcap_{k \in \N} \Sigma[s_0,s_1,\dots,s_k]
\]
and it turns out to be a unique number $\theta \in [0,\pi]$. we
can then write $\theta=[s_0;s_1,\dots]$ and call the sequence
$\{s_k\}$ the \emph{hexagon Farey expansion of $\theta$}.

\subsubsection{Direction recognition}

The problem posed and solved in this section is the following: is
it possible to recover the direction of a non-periodic trajectory
given its cutting sequence?

As we saw, we can associate to each direction an itinerary
under the Farey map and conversely to each infinite sequence a
unique direction it comes from. We will now relate such an
itinerary with the cutting sequence of the trajectory.

Again, we won't write down the proofs, since they are exactly the
same as those in the case of the octagon, which can be found in
\cite{octagon}.

We will see shortly how useful is the hypothesis of non
periodicity.

Let us consider a cutting sequence $w=c(\tau)$. We saw that a cutting
sequence always represents a path in at least one of the diagrams
$\mathscr D_{d_0}$. However, we want such $d_0$ to be well
defined and hence we suppose that the cutting sequence $w=w_0$ is
admissible in one unique diagram. We now consider the normal form
of $w$ as defined in \ref{normalword} and set $w_1=n(w)'$. Let us now
suppose again that $w_1$ is admissible in only one diagram
$\mathscr D_{d_1}$ and define in the same way $w_2$.

Applying the procedure recursively, we have defined the sequence of
words $\{w_k\}_{k \in \N}$ as
\[
\begin{cases}
w_0=w \\
w_{k+1}=n(w_k)'
\end{cases}
\]
and the sequence of admissible diagrams $\{d_k(w)\}_{k \in \N}$.

At each step we are assuming that the sequence is admissible in
one unique diagram and that seems to be a quite restrictive
hypothesis. Still, the condition can be proved to be quite simple
and is actually exactly the condition we had at the beginning of
the section: it can be proved that it is sufficient to ask that
the trajectory of cutting sequence $w$ is non-periodic.

This construction is what allows us to relate a cutting sequence
with the Farey expansion of the direction of our trajectory, still supposed until
now to be unknown. In fact, we have

\begin{theo}
If $w$ is a non periodic cutting sequence with sequence of
admissible diagrams $\{d_k(w)\}_{k \in \N}$, it is the coding of a
trajectory in direction
\[
\theta=[d_0(w);d_1(w),d_2(w),\dots].
\]
\end{theo}

It remains to give a simple outline of what happens in the
periodic case. The periodic case is based on the fact that
periodic cutting sequences come exactly from periodic trajectories
and they are related to the case of some directions for which we
have an ambiguity in writing down the Farey expansion, either
terminating with 1's or 5's.

\subsection{Characterization}

In this section we will introduce the full characterization of
cutting sequences in the set of all sequences $\alfab$.
Unfortunately, a straightforward characterization of the closure
of cutting sequences as the set of infinitely derivable sequences
(true in the case of the square) does not hold anymore. We will
cite the additional condition in two different ways: the former
following \cite{octagon} and introducing the generation rules and
the latter replacing them with the better know and more studied
substitutions, which will be in a smaller number than the
generation rules.

The generation is a combinatorial operation on words consisting of
interpolating letters in a word and it wants to be an inverse
operation to the derivation. We will denote this operator
$\mathfrak g_i^j$ and applied to a word $w$ admissible in diagram
$\mathscr D_i$ it will give us a word $W$ admissible in diagram
$\mathscr D_j$ and whose derivative is $W'=w$.

The first generation operator $\mathfrak g_i^0$ is obtained from
the diagrams:

\begin{align*}
\mathscr D_1 \qquad  &\xymatrix{ C \ar@/^1pc/[r] & A \ar@/^1pc/[l]
\ar@/^1pc/[r]^{\textbf{CB}}
 & B \ar@/^1pc/[l]^{\textbf{BC}} \ar@(ur,dr) }
&\mathscr D_2 \qquad  &\xymatrix{ C \ar@/^1pc/[r]^{\textbf{B}} & B
\ar@/^1pc/[l]^{\textbf{B}} \ar@/^1pc/[r]^{\textbf{BC}}
 & A \ar@/^1pc/[l]^{\textbf{CB}} \ar@(ur,dr)^{\textbf{CBBC}} } \\
\mathscr D_3 \qquad  &\xymatrix{ B \ar@/^1pc/[r]^{\textbf{B}} & C
\ar@/^1pc/[l]^{\textbf{B}} \ar@/^1pc/[r]^{\textbf{BBC}}
 & A \ar@/^1pc/[l]^{\textbf{CBB}} \ar@(ur,dr)^{\textbf{CBBC}} }
&\mathscr D_4 \qquad  &\xymatrix{ B \ar@/^1pc/[r]^{\textbf{C}} & A
\ar@/^1pc/[l]^{\textbf{C}} \ar@/^1pc/[r]^{\textbf{CBB}}
 & C \ar@/^1pc/[l]^{\textbf{BBC}} \ar@(ur,dr)^{\textbf{BB}} } \\
\mathscr D_5 \qquad  &\xymatrix{ A \ar@/^1pc/[r]^{\textbf{C}} & B
\ar@/^1pc/[l]^{\textbf{C}} \ar@/^1pc/[r]
 & C \ar@/^1pc/[l] \ar@(ur,dr)^{\textbf{BB}} } \\
\end{align*}

in the following way. Whenever we have a word admissible in
diagram $k$ we apply the operator $\mathfrak g_k^0$ by following
the path represented by the word in the corresponding diagram and
interpolating between letters the labels on each arrow.
Explicitly, if in the original word we have the transition
$L_1L_2$ and in the corresponding diagram the arrow that goes from
$L_1$ to $L_2$ has the label $w_1 w_2 w_3 w_4$ the new word will
have the subword $L_1 w_1 w_2 w_3 w_4 L_2$ instead.

It can be proved that for a $w$ admissible in diagram $\mathscr
D_k$, the new word $W=\mathfrak g_k^0 w$ is admissible in diagram
$\mathscr D_0$ and verify $W'=w$.

\begin{defin}
In a similar way we can define all the operators as
\[
\mathfrak g_j^i w=\pi_i^{-1} \cdot \mathfrak g_j^0 w.
\]
\end{defin}

Again, applied to a word admissible in diagram $\mathscr D_j$ it
gives a word admissible in diagram $\mathscr D_i$ whose derivative
is the original word.

Then we have the following characterization:

\begin{prop}
The sequence $w$ is the cutting sequence of a trajectory in direction
$\theta=[s_0;s_1,\dots]$ if and only if
\[
w \in \bigcap_{k\in \N} \mathfrak g_{s_1}^{s_0} \mathfrak
g_{s_2}^{s_1} \dots \mathfrak g_{s_{k-1}}^{s_{k-2}}\mathfrak
g_{s_k}^{s_{k-1}} u
\]
where $u$ is a word admissible in diagram $\mathscr D_{s_k}$.
\end{prop}

This same characterization can be restated in terms of
substitutions for both the hexagon and the octagonal case. It is
useful because they are more used and especially because they
allows us to give just five rules that can be applied in all cases
instead of five different diagrams for each of the six admissible
possibilities.

The idea is to notice that all diagrams have the same "shape" and
to give a label to each arrow in the shape as follows:
\[
\xymatrix{ \cdot \ar@/^1pc/[r]^{w_1} & \cdot
\ar@/^1pc/[l]^{\overline w_1} \ar@/^1pc/[r]^{w_2}
 & \cdot \ar@/^1pc/[l]^{\overline w_2} \ar@(ur,dr)^{w_3=\overline w_3}}
\]
where the conjugation means to read the label on the arrow from
the right to the left which is coherent with the labels of augmented diagrams.

Each arrow of a diagram corresponds to a transition in the word
represented by a path in such diagram and provided that from a
word admissible in diagram $\mathscr D_i$ we want a word
admissible in diagram $\mathscr D_j$ whose derivative is the
original word, we will define the substitution $\sigma_i$ on each
arrow $w_1$, $w_2$, $w_3$ recording how the interpolated word of
the diagram $\mathscr D_i$ gives a path in the diagram $\mathscr
D_j$. In other words, if an arrow $w_k$ links two letters $L_1$
and $L_2$ and carries the label $u_1 u_2$ we will record in
$\sigma_i(w_k)$ the "name" of the arrows corresponding in the
diagram $\mathscr D_j$ to the transitions $L_1 u_1$, $u_1 u_2$,
$u_2 L_2$. On the conjugated arrows they act as the same sequence
of arrows, but read from the right to the left and with each
element conjugated.

It can be seen that those substitutions are universal because
applying permutations of the letters does not interfere and hence
they can be obtained from the diagrams below.

A quick calculation show that they are

\begin{align*}
\sigma_1&
\begin{cases}
\sigma_1(w_1)=\overline w_1 \\
\sigma_1(w_2)=w_1 w_2 w_3 \\
\sigma_1(w_3)=w_3 \\
\sigma_1(\overline w_2)=\overline w_3 \overline w_2 \overline w_1 \\
\sigma_1(\overline w_1)=w_1 \\
\end{cases}
& & \sigma_2
\begin{cases}
\sigma_2(w_1)=w_2 w_3 \\
\sigma_2(w_2)=w_3 \overline w_2 \overline w_1 \\
\sigma_2(w_3)=w_1 w_2 w_3 \overline w_2 \overline w_1 \\
\sigma_2(\overline w_2)=w_1 w_2 w_3 \\
\sigma_2(\overline w_1)=w_3 \overline w_2 \\
\end{cases} \\
\sigma_3&
\begin{cases}
\sigma_3(w_1)=w_3 \overline w_2 \\
\sigma_3(w_2)=w_2 w_3 \overline w_2 \overline w_1 \\
\sigma_3(w_3)=w_1 w_1 w_3 \overline w_2 \overline w_1 \\
\sigma_3(\overline w_2)=w_1 w_2 w_3 \overline w_2 \overline w_1 \\
\sigma_3(\overline w_1)=w_2 w_3 \\
\end{cases}
& & \sigma_4
\begin{cases}
\sigma_4(w_1)=\overline w_2 \overline w_1 \\
\sigma_4(w_2)=w_1 w_2 w_3 \overline w_2 \\
\sigma_4(w_3)=w_2 w_3 \overline w_2 \\
\sigma_4(\overline w_2)=w_2 w_3 \overline w_2 \overline w_1 \\
\sigma_4(\overline w_1)=w_1 w_2 \\
\end{cases} \\
\sigma_5&
\begin{cases}
\sigma_5(w_1)=w_1 w_2 \\
\sigma_5(w_2)=\overline w_2 \\
\sigma_5(w_3)=w_2 w_3 \overline w_2 \\
\sigma_5(\overline w_2)=w_2 \\
\sigma_5(\overline w_1)=\overline w_2 \overline w_1 \\
\end{cases} & &
\end{align*}

The characterization becomes the following

\begin{prop}
The sequence $w$ is the cutting sequence of a trajectory in direction
$\theta=[s_0;s_1,\dots]$ if and only if there exists $w^{(k)}$ a
sequence of arrows in $\{w_1,w_2,w_3,\overline w_2, \overline
w_1\}$ such that
\[
w \in \bigcap_{k\in \N} \tau_{s_0} \sigma_{s_0} \sigma_{s_1} \dots
\sigma_{s_k} w^{(k)}
\]
where $\tau_{s_0}$ is the transformation of a sequence of arrows
back in letters by following the path they form in $\mathscr
D_{s_0}$.
\end{prop}

\subsection{Teichm\"uller disk} \label{teich}

\subsubsection{Marked translation surfaces and its actions}

The derivation on cutting sequences is more than a mere
combinatorial operator on it. As we will explain in this section,
we can somehow identify the deformations of a given translation
surface with the tangent bundle of the hyperbolic space (the
Teichm\"uller disk), on which we will define a cutting sequence
for a trajectory with respect to a tiling induced by a fundamental
domain for the action of the Veech group on it. Then, we will
relate the derivation previously defined and the new kind of
cutting sequences.

First of all we need hence to explain what space are we going to
work on. Given $S$ translation surface, we can consider the space
of \emph{marked translation surfaces} as the space of translation
surfaces correlated to $S$ as follow:

\begin{defin}
If $S$ is a translation surface and $f \colon S \to S'$ is an
affine diffeomorphism from $S$ to the translation surface $S'$ such that
the area of $S$ is equal to the area of $S'$ we say that $S'$ is
marked by $S$ and identify the triples $f,S,S'$ simply with $[f]$
with the convention that the function determines its domain and
codomain.
\end{defin}

\begin{defin}
We will also say that $[f]=f \colon S \to S'$ and $[g]=g \colon S
\to S''$ are affinely equivalent if there exists a translation
equivalence $[h]=h \colon S' \to S''$ such that $g=h \circ f$ and
we will write $[f]_A$ for the affine equivalence class of the
triple $[f]$.

Moreover, $[f]$ and $[g]$ are equivalent up to isometry if there
exists an isometry $[h]=h \colon S' \to S''$ such that $g=h \circ
f$ is an isometry and we will write $[f]_I$ for the isometry
equivalence class of the triple $[f]$.
\end{defin}

We will denote $\mathscr M_A(S)$ the set of translation surfaces
marked by $S$ up to affine equivalence and $\mathscr M_I(S)$ the
set of translation surfaces marked by $S$ up to isometry.

A simpler description of them can be made considering the
following:

\begin{prop}\label{ident}
The set $\mathscr M_A(S)$ can be identified with $SL_\pm(2,\R)$
and $\mathscr M_I(S)$ is isomorphic to $\H \simeq \D$.
\end{prop}

Here we are considering $\H$ as the upper half plane $\H=\{z \in
\C \mid \Im z>0\}$ and $\D=\{z \in \C \mid |z| <1\}$, identified
to each other by $\phi \colon \H \to \D$ defined as
$\phi(z)=\frac{z-i}{z+i}$.

We have two important actions on $\mathscr M_A(S)$ seen as the
group of matrices. The first one we want to talk about is the left
action of $SL_\pm(2,\R)$ and in order to describe it, it is useful
to give the explicit way of associating to a matrix a marked
translation surfaces and the converse map, that is actually
something we already talked about.

Indeed, we explained in \ref{Veech} that if we have a matrix $\eta
\in SL_\pm(2,\R) \subset GL(2,\R)$ and a translation surface $S$, then
we have a canonical map $\Phi_\eta \colon S \to \eta \cdot S$ such
that $D\Phi_\eta=\eta$ and exactly that triple $[\Phi]$ is our
marked translation surface in $\mathscr M_A(S)$. Conversely, to
each map $f \colon S \to S'$ we associated the matrix $Df$.

Therefore the action of a new matrix $\nu \in SL_\pm(2,\R)$ on a
class of translation surfaces $[f]$ is the composition of the
canonical map $\Psi_\nu \colon S' \to S''$ with $f$, obtaining
$\Psi_\nu \circ f \colon S \to S''$. Using the identification in
Proposition \ref{ident} and hence representing the set of marked
translation surfaces as a space of matrices, the action is
naturally the mere multiplication between matrices, applied on the
left.

We have a second action on $\mathscr M_A(S)$ which is the action
of $\aff(S)$. For an element $\Psi \colon S \to S$ in $\aff(S)$ it
corresponds to the composition $f \Psi \colon S \to S'$. Again,
using the identification in \ref{ident}, that action is the
multiplication on the right by the matrix $D\Psi$ representing the
affine automorphism $\Psi$.

As one can easily see considering the actions as multiplications
between matrices, the two actions naturally commute.

Let us now consider the action on $\mathscr M_A(S)$ of the
particular subgroup $g_t$ of $SL(2,\R)$ defined by

\[
g_t=\begin{pmatrix}
e^{\frac{t}{2}} & 0 \\
0 & e^{-\frac{t}{2}}
\end{pmatrix}.
\]

That particular action is called the Teichm\"uller flow and it can
be proved that if we project $\mathscr M_A(S)$ in $\mathscr
M_I(S)$ and use again the identification \ref{ident} it projects
to the hyperbolic geodesic flow in the best way we can expect. In
fact $g_t$-orbits in $\mathscr M_A(S)$ projects to unit speed
parametrized geodesics in $\H$ and we will call them
\emph{Teichm\"uller geodesics}.

We can explicitly give such an action writing down that the orbit
of a marked translation surface $\Phi_\nu \colon S \to \nu \cdot
S$ is the geodesic formed by the marked translation surfaces
$[\Phi_{g_t \nu}]$ as $t \in \R$.

Moreover, let us now consider the map from $\mathscr M_A(S) \cong
SL_\pm(2,\R)$ to $T_1\H$ defined as follows: to a triple
$\nu=[f]_A=\Phi_\nu \colon S \to \nu \circ S$ we associate the
pair made of the point of $\H$ given by the isometry class of the
element $[f]_I=p_f$ and of the unit tangent vector in such point to
the geodesic obtained by projecting the orbit of $[f]_A$ under the
$g_t$-action as explained before, i.e. the element $([f]_I=p_f,
D\Phi_{g_t \nu} \mid _{t_0})$ where $t_0$ is the time of the flow
such that $g_{t_0}=p_f$.

Such a map is a surjective 4 to 1 map. For simplicity, we will
firstly identify the two copies of the upper and lower half plane
$\H^+$ and $\H^-$ and then make the quotient with the Kernel of
the new map $\{\pm \id \}$ and hence consider the marked
translation surfaces as $PSL(2,R)$ obtaining a 1 to 1 map in $T^1
\H$ and hence an identification of the two spaces.

We now return to our case of the surface obtained from the hexagon
$S_E$ and consider the marked translation surfaces affinely
equivalent to $S_E$, identified with $T^1 \D$ and by the
correspondence with $PSL(2,\R)$. The centre of the disk will be
identified with the identity matrix and hence with the triple $\id
\colon S_E \to S_E$. As we saw before, we have the action of the
Veech group on it and it turns out that we can explicitly give a
fundamental domain for such an action.

\begin{figure}[h]
\centering
\includegraphics[width=0.6\textwidth]{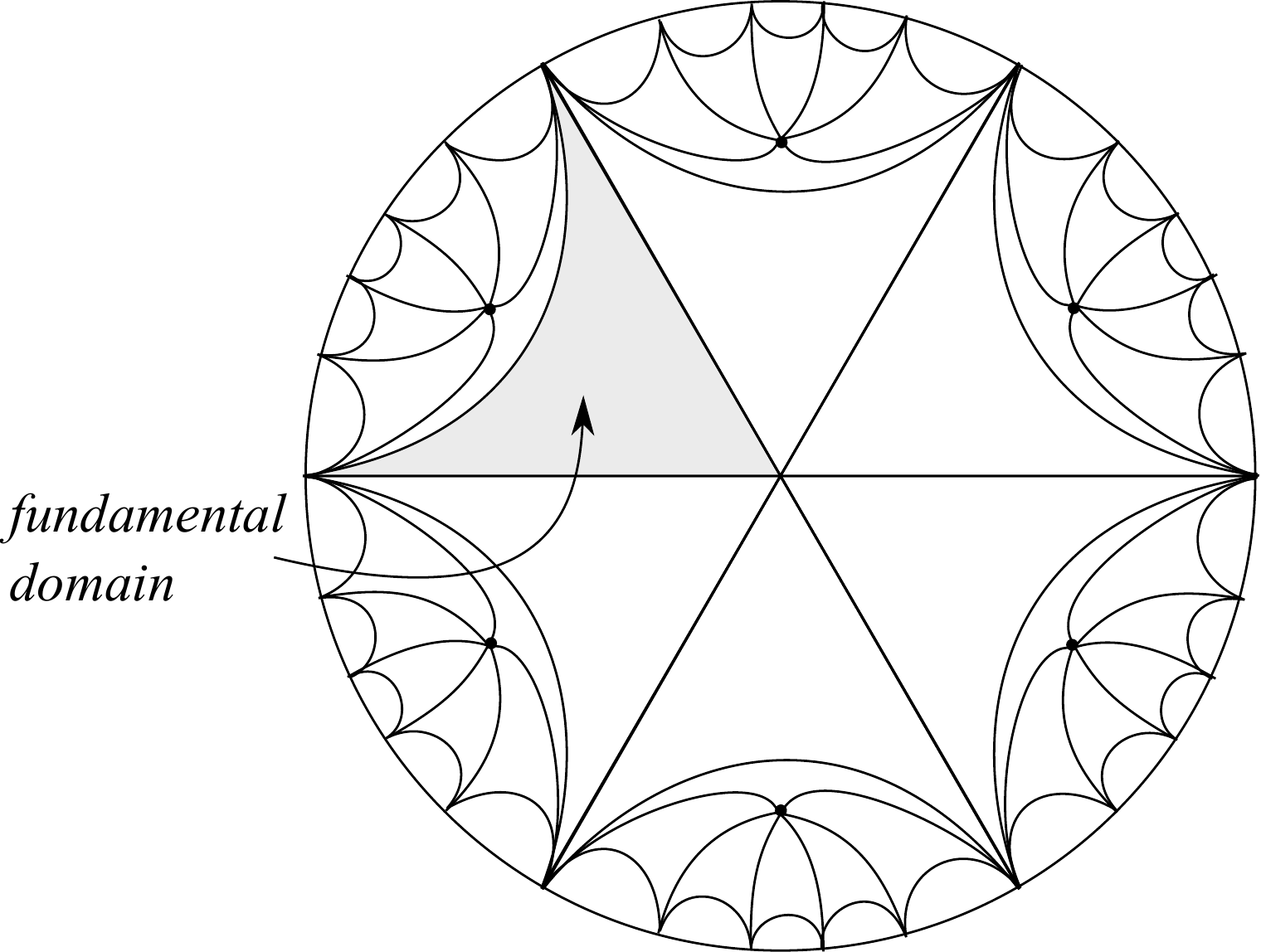}
\caption{Fundamental domain tessellation} \label{fundtiling}
\end{figure}

Let us consider the hyperbolic triangle formed by the horizontal
line between the centre of the disk and the extreme left point of
the boundary, the line between the boundary and the centre forming
angle $\frac{\pi}{3}$ with the previous one and the third geodesic
joining the two boundary points as in Figure \ref{fundtiling}. We
will call it $\mathscr F$ and we affirm that it is a fundamental
domain for the action of the Veech group. Indeed, the three
generators of such group $\alpha, \beta, \gamma$ are hyperbolic
reflections with respect to the three sides of $\mathscr F$ and
hence the Veech group is the triangular group $\Delta
(3,\infty,\infty)$.

In fact if we apply on the left (being a left action) $\alpha$ to
an element of $\H$ seen as a pair and then send it in $\D$ with
the map describe before, we have that point reflected in the disk
with respect to the horizontal side. In the same way, applying
$\beta$ means reflecting with respect to the geodesic passing for
the centre of the disk and forming angle $\frac{\pi}{3}$ with the
horizontal left oriented line and $\gamma$ is the reflection with
respect to the third side.

The images under the Veech group of the fundamental domain give a
tiling of the Teichm\"uller disk.

\subsubsection{Teichm\"uller geodesics and hexagonal tiling}

We will now introduce a special type of geodesics in the
Teichm\"uller disk and a second tiling on it and explain how we
can code the flow of such a geodesic with respect to such a
tiling. In the next paragraph, we will see how that is related
with the coding of trajectories in the surface $S_E$ and
particularly with the derivation.

Let us fix a direction $\theta$, and think about it as the
direction of a trajectory $\tau$ in $S_E$. We can hence consider
the subgroup conjugate to the geodesic flow introduced before

\[
g_t^\theta= \rho^{-1}_{\frac{\pi}{2}-\theta} \circ g_t \circ
\rho_{\frac{\pi}{2}-\theta}.
\]

Geometrically, as $t$ increase, since $g_t$ expands the horizontal
direction and contracts the vertical one, and the conjugation with
the counterclockwise rotation of angle $\frac{\pi}{2}-\theta$
sends the direction $\theta$ in a vertical position and sends it
back, such a subgroup contracts the direction $\theta$ and expands
the orthogonal direction.

Let us now call \emph{Teichm\"uller geodesic ray}
\[
\tilde r_\theta=\{g_t^\theta \cdot S_E\}_{t \geq 0}
\]
which is a geodesic ray in $T_1\D$ and denote $r_\theta$ its
projection on $\D$. Naturally, $r_\theta$ is a geodesic ray in
$\D$, starting in the centre of the disk and converging to a point
of the boundary. The mere calculation of the matrix
\[
g_t^\theta=
\begin{pmatrix}
e^t \sin^2 \theta +e^{-t}\cos^2 \theta & -e^t \sin \theta
\cos \theta +e^{-t}\sin \theta \cos \theta \\
-e^{t}\sin \theta \cos \theta e^{-t} \sin \theta \cos \theta & e^t
\cos^2 \theta +e^{-t} \sin ^2 \theta
\end{pmatrix}
\]
and its action on $i$ (which means imposing the geodesic to pass
for $i$), shows that the limit as $t \to \infty$ is the boundary
point $-\tan \theta \in \partial \H$ which is sent by the
identification of $\H$ and $\D$ in the boundary point
$e^{i(\pi-2\theta)}$ since it is

\begin{align*}
-(\cos^2 \theta - \sin^2 \theta)+2i\cos \theta \sin \theta &=-\cos
2\theta +i \sin 2\theta= \\
&=cos(\pi-2\theta)+i \sin(\pi-2\theta)= \\
&=e^{i(\pi-2\theta)}.
\end{align*}

This makes it easier to understand how the angle $\theta$
parametrizes the boundary and equivalently the geodesic rays
coming out from the centre of the disk: the ray $r_0$ is the one
corresponding to the horizontal line oriented on the left and the
generic ray $r_\theta$ is the one which form a clockwise angle of
$2\theta$ with the ray $r_0$. This means that as $\theta$ run
along the interval that we always considered up to apply a
$\pi$-rotation, it parametrize all the boundary of the disk.

If we now return to consider the image of the fundamental domain
$\mathscr F$ under the Veech group, we can naturally act with the
six elements $\nu_i$ explained in the first section. We call $E_0$
the side connecting the two boundary points and we can define
$E_1, \dots, E_5$ as $E_i=E_0 \nu_i$ meaning the right action of
$\nu_i$ on each point composing the side. We hence have an
hyperbolic hexagon $\mathscr E$ limited by the sides $E_0, \dots
E_5$ and by acting on it with the Veech group we obtain a new
tessellation of the disk $\D$ the \emph{ideal hexagon
tessellation}, obtained by the previous one deleting the tile
sides joining an image of the centre of the disk, i.e. consider as
one unique tile all six triangles sharing one common vertex as in
Figure \ref{hextiling}.

\begin{figure}[h]
\centering
\includegraphics[width=0.6\textwidth]{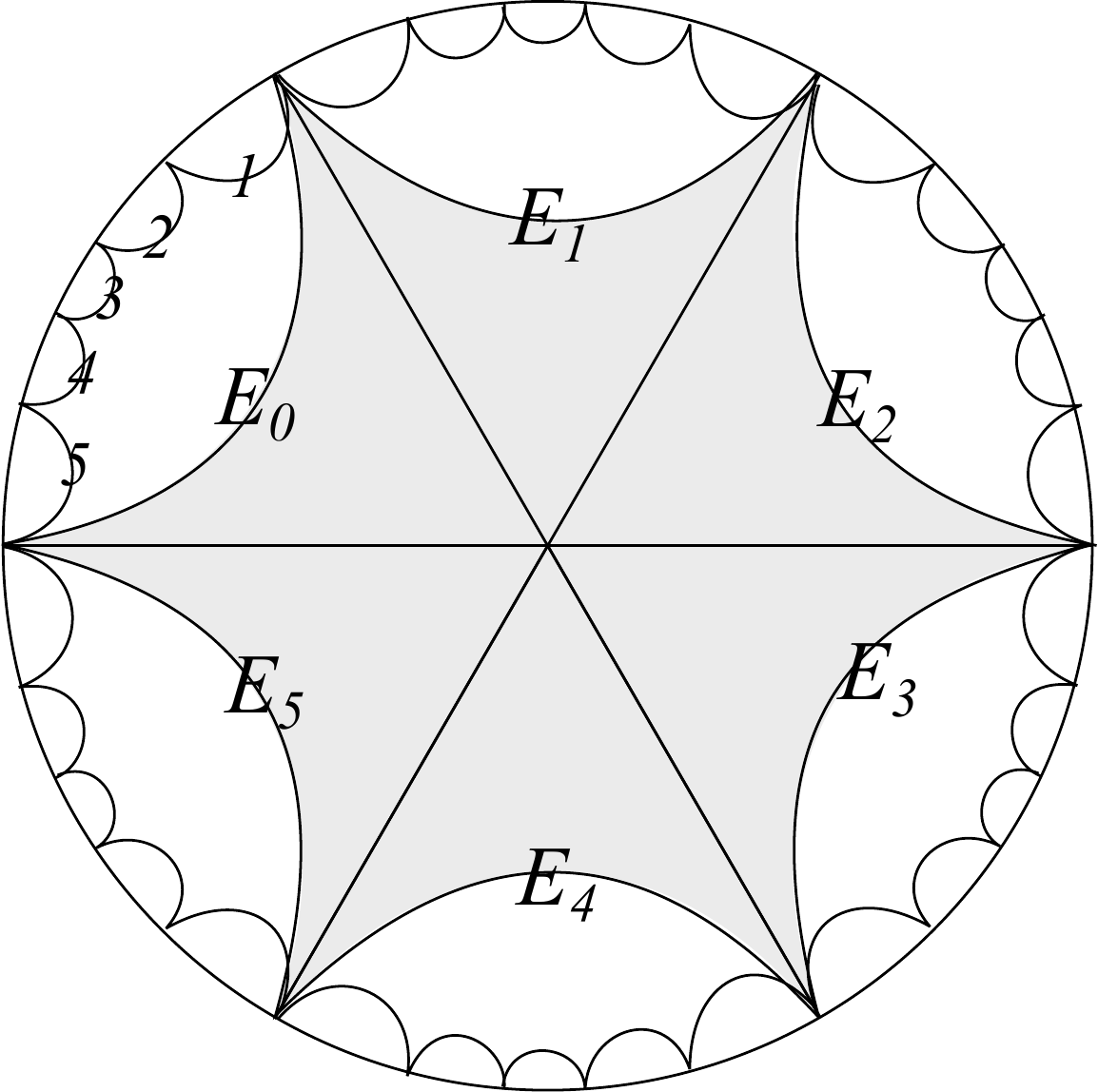}
\caption{Ideal hexagons tessellation} \label{hextiling}
\end{figure}

\begin{rk}\label{angolo}
For the parametrization of the boundary with the angle as
described before, we have that a Teichm\"uller geodesic ray
$r_\theta$ crosses the side $E_i=E_0 \nu_i$ if and only if $\theta
\in \Sigma_i=\nu_i^{-1}\Sigma_0$.
\end{rk}

Such an ideal hexagon tessellation allows us to construct a tree
related to it. It will be useful to explain the propositions in
the following section.

The set of vertices of such a tree is the set of centres of the
ideal hexagons, i.e. the images under the action of the Veech
group of the centre of the disk. The edges will be geodesic
segments connecting two centres if and only if the two hexagons
share a common side. Naturally, we can consider such a tree not
embedded in the hyperbolic disk and have segments joining the
vertexes. We will have hence a regular tree $\mathscr T$ with six
edges coming out of each vertex.

To each Teichm\"uller geodesic ray we can associate a path in the
tree $\mathscr T$ that we will call \emph{combinatorial geodesic
$P_\theta$ approximating $r_\theta$}. If $r_\theta $ does not
terminate in a vertex of the tessellation (and it can be proved
that it corresponds to the case of periodic trajectories that we
didn't consider since the beginning), it intersects an infinite
number of sides of ideal hexagons. We associate to it a path
starting from the centre of the disk and crossing all the couple
of vertexes which are the two centres of the two hexagons which
share the side crossed by $r_\theta$.

\subsubsection{Teichm\"uller cutting sequences}

By giving a label to each segment of the tessellation with ideal
hexagon, we can assign to each Teichm\"uller geodesic a
\emph{Teichm\"uller cutting sequence} by coding the intersections
of the geodesic with the labels assigned to each side. It turns
out that there is an intimate connection between the derivation
that we used in the first part of the work on cutting sequences on
trajectories in the surface $S_E$ and the Teichm\"uller cutting
sequences.

This is a similar procedure to the one described by Series in
\cite{squarehyp}.

In order to label the sides of the tiling, we introduce the
subgroup of the Veech group $V_{\mathscr E}$. It is the subgroup
generated by the hyperbolic reflections with respect to the sides
of $\mathscr E$, that can be obtained by conjugating the reflection
with respect to the side $E_0$ (which is our matrix $\gamma$) by
the matrices $\nu_i$ introduced at the beginning of the work.
$V_{\mathscr E}$ in hence the subgroup generated by
$\gamma_i=\nu_i^{-1} \gamma \nu_i$ for $i=1, \dots 5$.

The hexagon $\mathscr E$ is a strict fundamental domain for such a group and
this allows us to label each side of the tiling in the following
way: given a side, there is an element of $V_{\mathscr E}$ that
sends it back in a side $E_i$ of $\mathscr E$ and we label such a
side of the tiling with the number $i \in \{0,\dots,5\}$.

We can also use it to give a labelling on the tree $\mathscr T$ by
labelling each edge with the number of the side of the tiling
transversal to it.

\begin{rk} \label{idealhex}
For each side shared by two ideal hexagons, if it is labeled with
$i$, it will exists $n \in \N$ and a sequence
$\{s_k\}_{k=0,\dots,n} \in \{0,\dots,5\}$ such that one ideal
hexagon is $\mathscr E \gamma_{s_n} \gamma_{s_{n-1}} \dots
\gamma_{s_1} \gamma_{s_0}$ and the other one is $\mathscr E
\gamma_i \gamma_{s_n} \gamma_{s_{n-1}} \dots \gamma_{s_1}
\gamma_{s_0}$.
\end{rk}

Given a geodesic in the Teichm\"uller disk, we define the
\emph{Teichm\"uller cutting sequence} associated to it as the
sequence of labels of sides of the ideal hexagon tessellation
crossed by the geodesic. In particular, we can do that for a
Teichm\"uller geodesic ray $r_\theta$ and associate to it
$c(r_\theta) \in \{0,\dots,5\}^{\N}$. It represents also the
labels of the edges forming the path $p_\theta$ in $\mathscr T$
correspondent to the ray $r_\theta$.

\begin{rk}
It is evident that since a cutting sequence never returns back, in
a Teichm\"uller cutting sequence the same label never occurs twice
in a row. Moreover, it can be seen that it is the only restriction
that one needs on an element of $\{0,\dots,5\}^\N$ for being a
Teichm\"uller cutting sequence.
\end{rk}

In exactly the same way as in \cite{octagonteich} we can prove the
relation between the Teichm\"uller cutting sequences of a ray
$r_\theta$ and the cutting sequence of a trajectory in direction
$\theta$,as in the following proposition, which gives a
geometrical interpretation of  the derivation on a cutting
sequence.

Let us consider $w$ cutting sequence of a trajectory $\tau$ in
direction $\theta$, its $k$-derived sequence $w^{(k)}$ and the
corresponding Teichm\"uller geodesic ray $r_\theta$ with
Teichm\"uller cutting sequence $c(r_\theta)$.

We define for each $k \geq 1$ the element of the Veech group
$\gamma^{(k)}$ as the composition of $\gamma_{s_{k-1}} \gamma_{s_{k-2}} \dots \gamma_{s_0}$ where $s_0, \dots s_{k-1}$ are
the first entries of the Teichm\"uller cutting sequence
$c(r_\theta)$. Such an element sends the $k$-th ideal hexagon
crossed by $r_\theta$ back in $\mathscr E$ by Remark
\ref{idealhex}. As we did before for the elements of the Veech
group, we have a correspondent affine automorphism
$\Psi_{\gamma^{(k)}} \colon S_E \to S_E$ whose derivative is
$\gamma^{(k)}$ as the composition of the canonical map
$\Phi_{\gamma^{(k)}} \colon S_E \to \gamma^{(k)} \cdot S_E$ and
the cut and paste map $\Upsilon^{(k)} \colon \gamma^{(k)} \cdot
S_E \to S_E$. On the original hexagon this automorphism consists
in applying an affine deformation and then cutting and pasting it
back in the original hexagon. Such deformation stretches the
direction $\theta$ in the hexagon.

\begin{prop}\label{fact3}
The $k$-th derived sequence $w^{(k)}$ of a cutting sequence $w$ of
a trajectory $\tau$ in direction $\theta$ in $S_E$ is the cutting
sequence of the same trajectory with respect to the sides of
$\Psi_{\gamma^{(k)}}E$ with the labelling induced by
$\Psi_{\gamma^{(k)}}$.
\end{prop}

\begin{proof}
Saying that $w^{(k)}$ is the cutting sequence of $\tau$ with
respect to $\Psi_{\gamma^{(k)}}E$ is equivalent to say that
$w^{(k)}$ is the cutting sequence of $\Psi_{\gamma^{(k)}}^{-1}
\tau$ with respect to $E$.

We prove it by induction. For the case $k=0$ it is just
$w=c(\tau)$.

For $k>0$ let's set

\[
\begin{cases}
\tau_0=\tau \\
\tau^{(k)}=\Psi_{\gamma^{(k)}}^{-1} \tau=\Psi_{\gamma_{s_{k-1}}}
\dots \Psi_{\gamma_{s_0}} \tau.
\end{cases}
\]

Our aim becomes to show that $w^{(k)}$ is the cutting sequence of
$\tau^{(k)}$ for each $k>0$.

We now assume as inductive hypothesis that $w^{(k)}$ is the
cutting sequence of $\tau^{(k)}$.

The trajectory $\tau^{(k)}$ has direction in $\Sigma_{s_k}$ for
Remark \ref{angolo}. This means that $\nu_{s_k} \tau^{(k)}$
has direction in $\Sigma_0$ and hence its cutting sequence is
$\pi_{s_k} \cdot w^{(k)}$ for \ref{commutation}. Let us now derive
once more in order to obtain $w^{(k+1)}$. We have
\[
(\pi_{s_k} \cdot w^{(k)})'=c(\Psi_\gamma \nu_{s_k}\tau^{(k)})
\]
for Proposition \ref{central}. Now we act with
$\nu_{s_k}^{-1}$ on the trajectory $\Psi_\gamma
\nu_{s_k}\tau^{(k)}$ so that it becomes exactly $\tau^{(k+1)}$ and
we have
\begin{align*}
c(\tau^{(k+1)})&=c(\Psi_{\gamma_{s_k}} \tau)=c(\nu_{s_k}^{-1}
\Psi_\gamma \nu_{s_k} \tau^{(k)})= \\
&=\pi_{s_k}^{-1} \cdot (\pi_{s_k}
\cdot w^{(k)})'=\pi_{s_k}^{-1} \cdot \pi_{s_k} \cdot
(w^{(k)})'=w^{(k+1)}
\end{align*}
using again \ref{commutation} and then \ref{invariance} and hence
the thesis.
\end{proof}

\section{Hexagon vs square}

In this section we will try to investigate two aspects of
differences and similarities between hexagon and square.

In the first part we point out that the derivation in the Series
method for the square is slightly different from the one used for the
hexagon here and for the octagon in \cite{octagon} and we show how to
adapt the second one to the case of the square too.

In the second part we will show where is the hexagon located in
the space of flat tori and we will construct a dictionary between
the hexagon and its representative as a parallelogram.

\subsection{Derivations in the square}\label{second}

At a first reading \cite{square} and \cite{octagon} seem to have
three different definitions of the operation of derivations

\begin{enumerate}
\item In the case of the square a cutting sequence
consists in blocks of length $n_0$ or $n_0+1$ of the same letter
separated from one occurrence of the other one. Series's
derivation in \cite{square} consists of erasing $n_0$ letters from
each block before interchanging the roles of the two letters.
\item In \cite{octagon}, the derivation used in the case of the
square consists in erasing one letter from each block until
possible before interchanging the roles of the two letters.
\item The method used here for the hexagon (as well as in
\cite{octagon} and \cite{pentagon} for regular octagons and double
regular pentagons) is to keep the sandwiched letters and drop the
others.
\end{enumerate}

\subsubsection{Series derivation}

The difference between 1. and 2. can be easily explained. In fact,
the Series derivation is nothing but an acceleration of the one
described by Smillie and Ulcigrai. This difference is related to
the fact that the $n$-th branch of the Gauss map used by Series is
just $G_n=F_1 \circ F_0^{n-1}$ on the interval
$[\frac{1}{n+1},\frac{1}{n}]$, as we verified by induction.

\subsubsection{Where sandwich derivation fails}

On the other hand, the difference between the Series derivation
and the sandwich derivation remains a real difference. Consider a
square $Q$ with opposite sides glued obtaining the surface $S_Q$
and try to follow the same procedure as for the hexagon. Since we
have one more symmetry we can immediately reduce to the case of a
trajectory $\tau$ in direction $\theta \in [0,\frac{\pi}{2}]$ and
by applying $\nu_1$, reflection with respect to the bisectrix of
the first quadrant and the permutation $\pi_1=(AB)$ we can always
send a direction in $[0,\frac{\pi}{4}]$. The cylinder
decomposition of the square in a unique cylinder given by the
square itself gives us the Veech element used by Series:
\[
\sigma=
\begin{pmatrix}
1 & 1 \\
0 & 1
\end{pmatrix}.
\]

\begin{rk}
We notice that applying the shear $\sigma$ or its
orientation-reversing counterpart $\gamma=\begin{pmatrix}
-1 & 1 \\
0 & 1
\end{pmatrix}$ in this case does not change anything because the
reflection with respect to the vertical axis does not change
labels.
\end{rk}

Unfortunately, if we try to characterize cutting sequences as
before we need to show that cutting sequences are infinitely
derivable, but if we try to reproduce the proof of Proposition
\ref{central} we do not get to the same conclusion.

In the first step we construct the augmented sequences $\tilde
c(\tau) \in \{A,B,c\}^\Z$ by recording on the diagram $\tilde
{\mathscr D_0}$ the intersection with the diagonal $c$:

\[
\tilde {\mathscr D_0} \qquad \qquad  \xymatrix{ A \ar@/^1pc/[r] &
B \ar@/^1pc/[l]  \ar@(ur,dr)^c }
\]

In the second step we interchange the role of $B$ and $c$
obtaining $\hat c(\tau) \in \{A,c\}^\Z$ and the diagram

\[
\tilde {\mathscr D_0'} \qquad \qquad  \xymatrix{A \ar@(ul,dl)_B
\ar@/^1pc/[r]^B & c \ar@/^1pc/[l]^B  \ar@(ur,dr)^B }
\]

In the third step we apply the cut and paste map $\Upsilon_Q$ that
act sending the polygon deformed by the Veech shear back in the
original polygon, but we find out that the since $A'=A$ and $B'=c$
the cutting sequence $\overline c(\tau)$ of the trajectory with
respect to the sides $A', B'$ of the new square $Q'=\sigma Q$ is
exactly the previous one $\hat c(\tau)$. As we see in Figure
\ref{square}, this is because the new side coincide with the
diagonal instead of being inscribed in a parallelogram formed by
previous sides and diagonals, which determined uniquely which of
the new sides are crossed.

\begin{figure}[ht]
\centering
\includegraphics[width=0.5\textwidth]{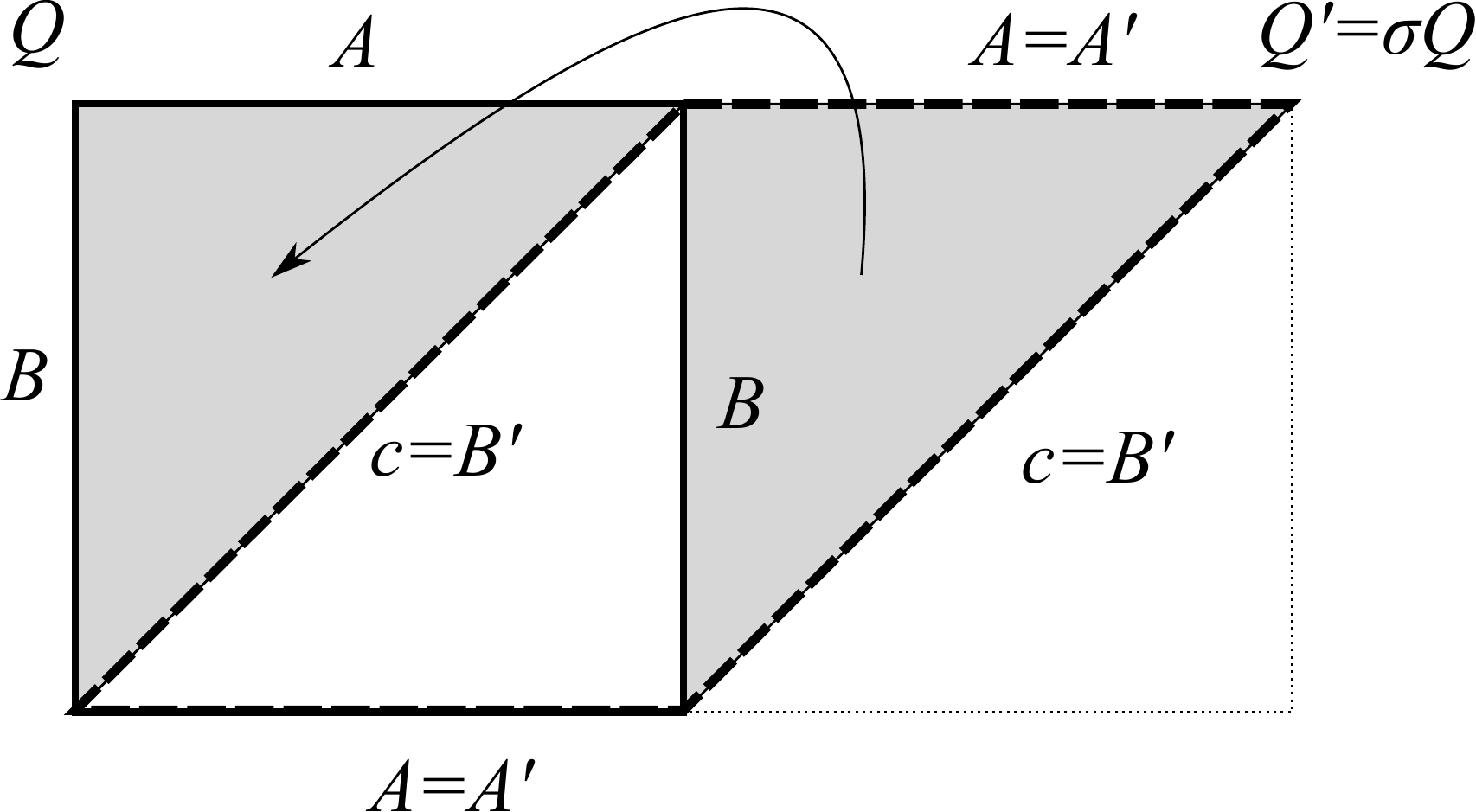}
\caption{Action of the Veech element and of the cut and paste map on the square} \label{square}
\end{figure}

\subsubsection{How sandwich derivation can work}

At this point we tried to understand if there is a way to adapt
the procedure used for the hexagon to make it work in the case of
the square too. It turns out that it is possible to make it work
just using a slightly different matrix.

In fact in the cases of $2n$-gons for $n \geq 3$ we always used as
Veech shears the matrices

\[
\begin{pmatrix}
1 & 2 \cot \left(\frac{\pi}{2n}\right) \\
0& 1
\end{pmatrix}
\]

which in this case is not the shear $\sigma$ given by the cylinder
decomposition but the composition

\[
\sigma'=\sigma^2=
\begin{pmatrix}
1 & 2\\
0& 1
\end{pmatrix}.
\]

Using this matrix instead of the previous one is the same than
considering the square rotated by $\frac{\pi}{4}$, obtaining a
diamond-shape and looking for an horizontal cylinder
decomposition. The decomposition showed in Figure \ref{diamond}
gives us a cylinder of modulus exactly 2.

\begin{figure}[h]
\centering
\includegraphics[width=0.5\textwidth]{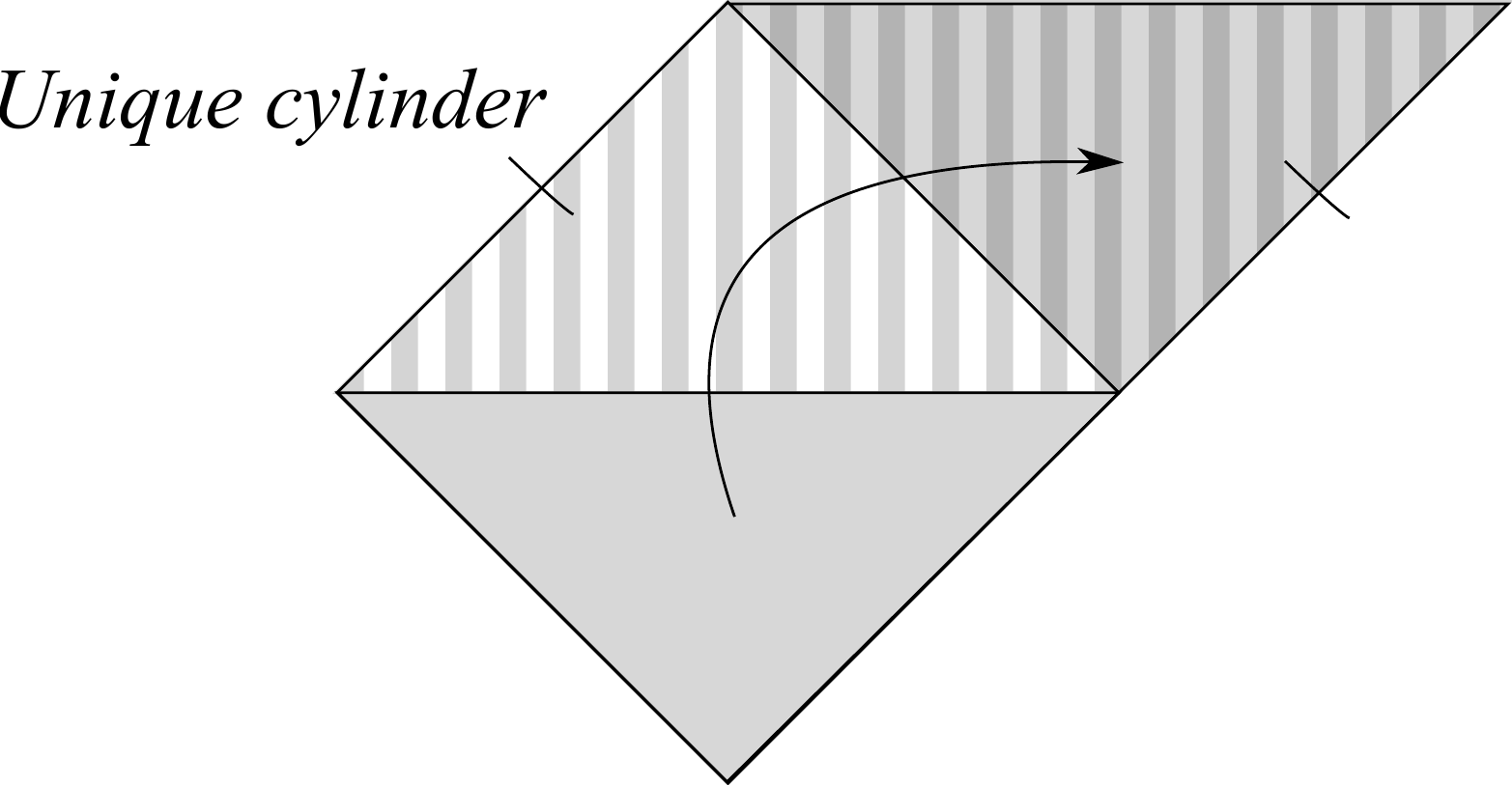}
\caption{Horizontal cylinder decomposition for the rotated square} \label{diamond}
\end{figure}

Using this matrix we can continue to reproduce the proof of Proposition \ref{central} and we have first and second step as
before.

For the third step, after applying the cut and paste $\Upsilon_Q'$
showed in Figure \ref{squareshear} the new side $B'$ is in this
case the diagonal of a parallelogram with sides $A$ and the
diagonal $c$ as we wanted.

\begin{figure}[ht]
\centering
\includegraphics[width=0.7\textwidth]{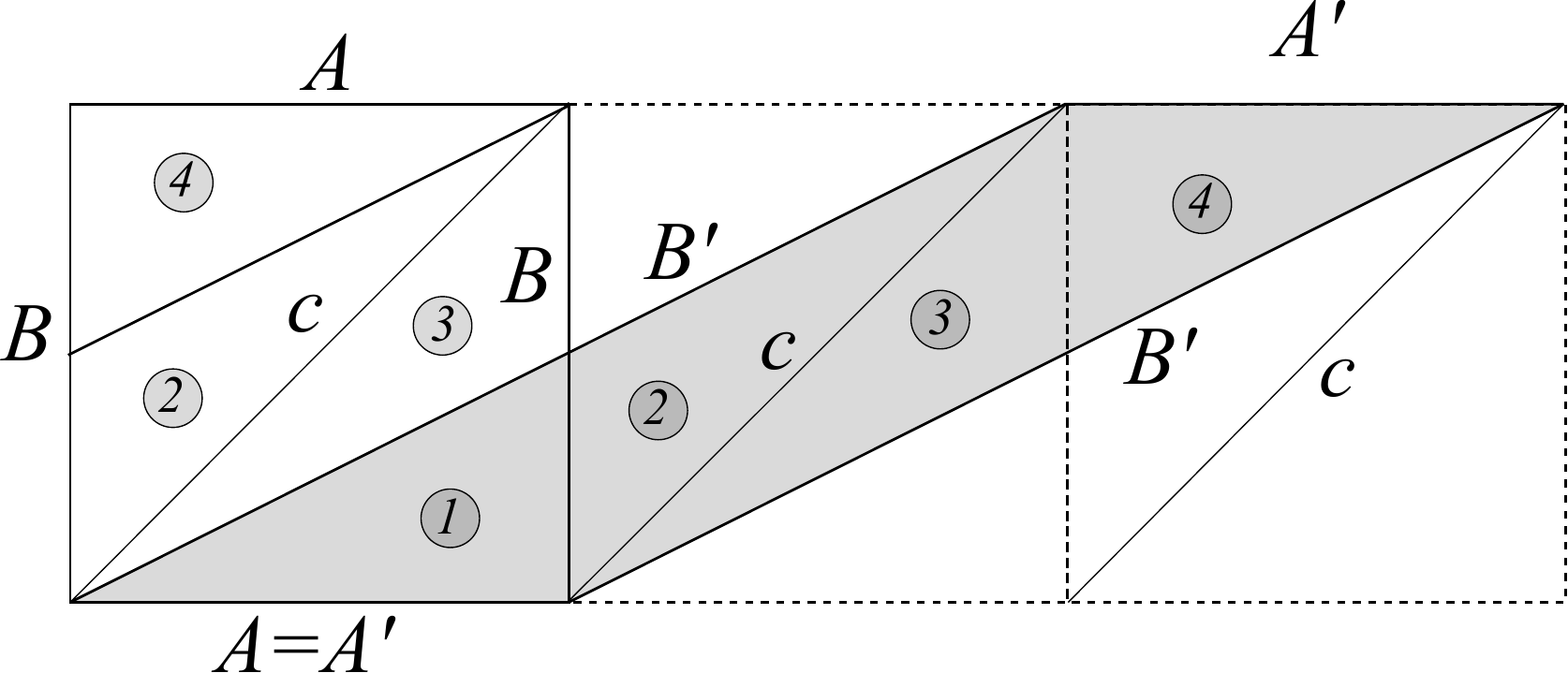}
\caption{Veech shear and cut and paste with the new matrix} \label{squareshear}
\end{figure}

This means that we can construct the diagram recording the
crossing of the sides of the new square as

\[
\mathscr D_0' \qquad \qquad  \xymatrix{A \ar@(ul,dl)_{B'}
\ar@/^1pc/[r] & c \ar@/^1pc/[l]  \ar@(ur,dr)^{B'} }
\]
\smallskip

and by looking at diagrams $\mathscr D_0 '$ and $\tilde {\mathscr
D_0'}$ as in the original proof we notice again that the
sandwiched letters are the one kept.

We can conclude that sandwich derivation gives us a different way
of characterizing the cutting sequences in a square. We already knew
that cutting sequences in the square are infinitely derivable in
the Series sense and we have now proved that cutting sequences in
the square are infinitely derivable also in the sense of sandwich
derivation.

Exactly as before, we can construct a Farey map different from the
previous one.

In the Teichm\"uller disk of the deformations of the square, using
the new matrix one can repeat exactly the same proofs as in the
case of the hexagon. The correspondent matrix
$\gamma'=\begin{pmatrix} -1 & 2
\\ 0 & 1 \end{pmatrix}$ is the hyperbolic reflection with respect
to the geodesic passing through 0 and 1 and hence we have the two
tessellations by hyperbolic triangles and by hyperbolic squares
exactly as before, by considering the two matrices

\begin{align*}
\alpha=\begin{pmatrix} 1 & 0 \\ 0 & -1 \end{pmatrix}, & &
\beta=\begin{pmatrix} 0 & 1 \\ 1 & 0 \end{pmatrix}
\end{align*}

which are, as in the case of the hexagon, the reflection with
respect to the horizontal axis and the matrix $\nu_1$.

Proposition \ref{fact3} can be also stated and proved in the
same way for the case of the square.

\subsection{Hexagon as a flat torus} \label{third}

\subsubsection{Space of flat tori}

The space of flat tori is the space of lattices in the plane. We
identify a lattice with its two-dimensional basis $\Z v_1+ \Z
v_2$. Conventionally, up to change $v_1$ and $v_2$ we can consider
$v_1$ to have smaller length. Up to rotating and scaling, we can
suppose $v_1$ to be coincident with the first vector of the
classic basis in the plane $e_1$ and the basis to be positively
oriented, so that $v_2$ is in the upper half plane that we
identify with $\H$. But we can restrict the domain even more,
because acting with an element of $SL(2,\Z)$ we obtain the same
lattice because we just change the length of a integer and hence
we can quotient the upper half plane by this group. This means the
following:

\begin{prop}
The space of flat tori, i.e. the space of lattices is
\[
\H /SL(2,\Z),
\]
where $SL(2,\Z)$ is considered acting on the right by M\"obius transformations.
\end{prop}

\begin{figure}[h]
\centering
\includegraphics[width=0.5\textwidth]{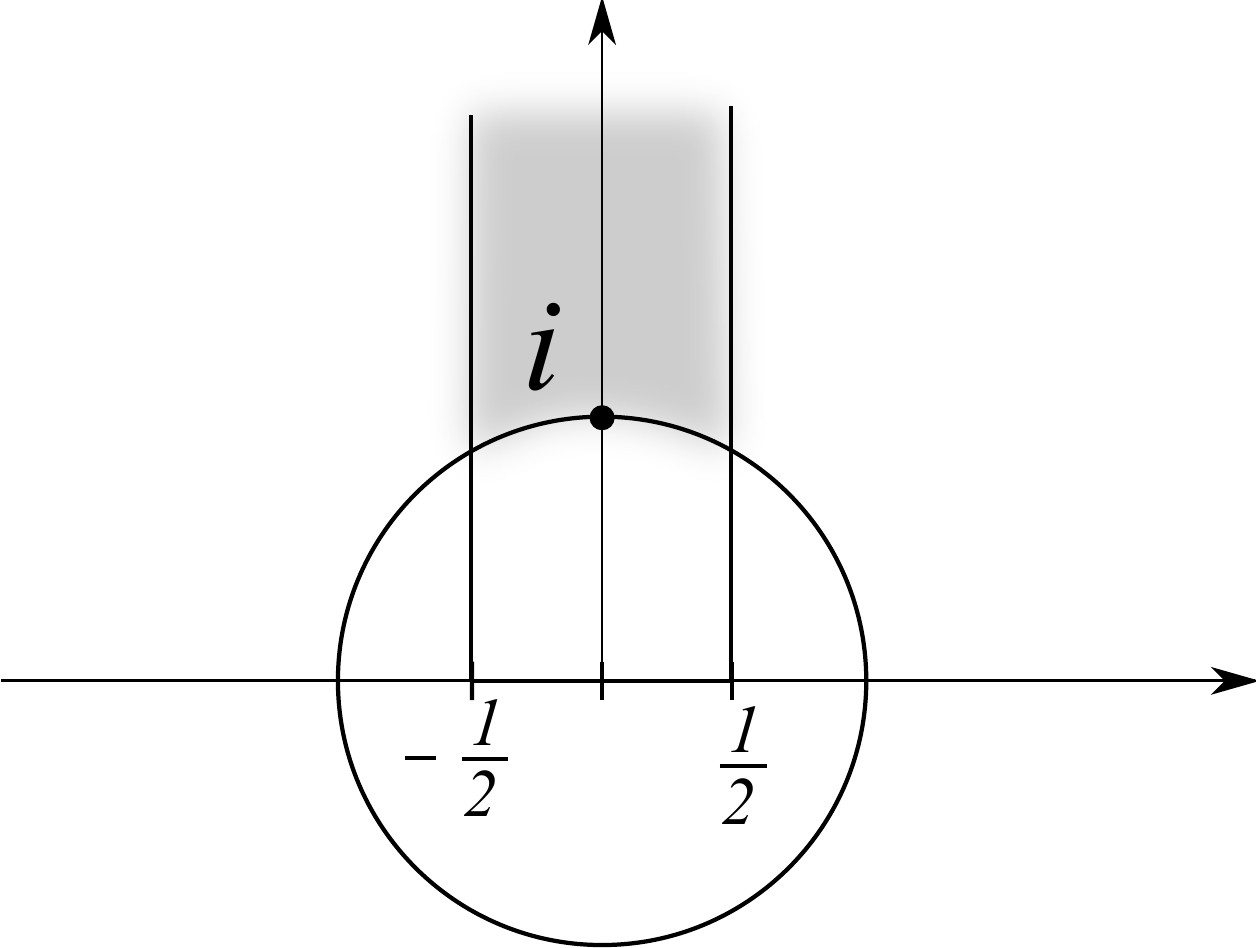}
\caption{The space of lattices} \label{flattori}
\end{figure}

It is well known that a fundamental domain for the $SL(2,\Z)$
action is the infinite region limited by the vertical lines through
$\frac{1}{2}$ and $-\frac{1}{2}$ and the upper semicircle of the
unit circle.

It is also known that the hexagon is represented by the
intersection point of the semicircle with the imaginary axis.

Naturally, this implies that we cannot find a cut and paste from
the hexagon to the square, even if they both represent flat tori.
This means also that with a cut and paste map we can send the
hexagon in a parallelogram with the bottom left angle of
$\frac{\pi}{3}$. In fact the modulus of such a parallelogram seen
as a cylinder is the same of the modulus of the unique cylinder of
the horizontal cylinder decomposition of the hexagon and it is
clear that the modulus of the parallelogram determines its
position in the space of flat tori.

\subsubsection{Dictionary}

Let us now add to the figure of an hexagon the diagonals $e$ and
$f$ starting from the bottom left angle and such that the first
one is vertical and the second one is in direction $\frac{\pi}{6}$
as in Figure \ref{hexparall}. We can cut the hexagon along these
diagonals and paste the two smaller pieces on the sides labelled
$B$ and $A$ of the hexagon, so that the two sides labelled $C$ are
glued together.

\begin{figure}[h]
\centering
\includegraphics[width=0.7\textwidth]{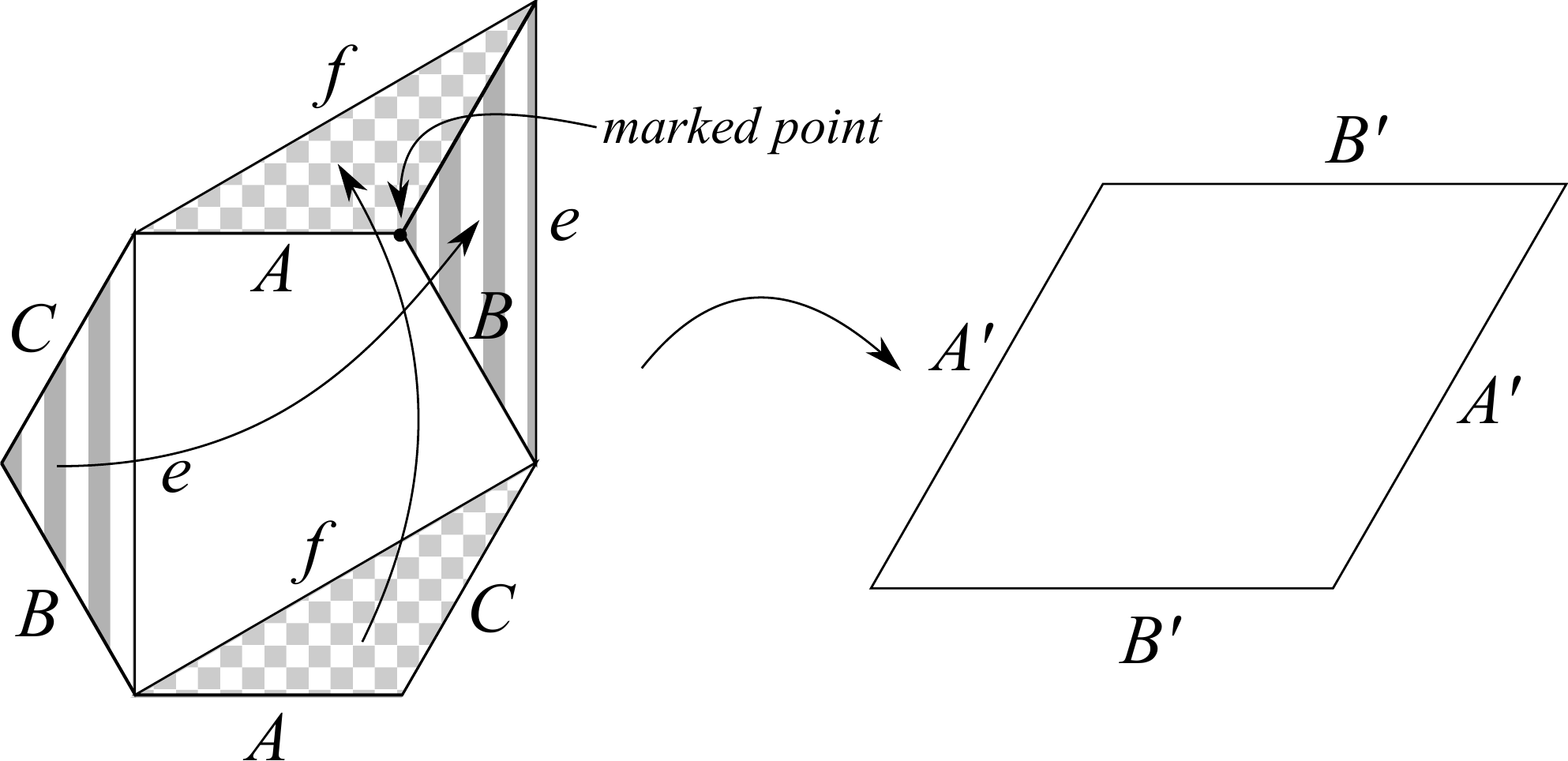}
\caption{Cut and paste from the hexagon to the parallelogram} \label{hexparall}
\end{figure}

The cut and paste map we just described is showed in Figure
\ref{hexparall}. Up to rotating the figure we have obtained the
requested parallelogram, equivalent to the hexagon and
representing it in the fundamental domain of $SL(2,\Z)$, which is
the space of flat tori, as we explained before.

Let us consider a trajectory in a direction belonging to
$\Sigma_0$ in the hexagon. It naturally is also a trajectory in
such a parallelogram $P$. Suppose we have the cutting sequence of
such a trajectory with respect to the sides of $E$. In order to
find out the cutting sequence with respect to the sides of $P$ we
first record the passage for the diagonals $e$ and $f$ in the
transition diagram $\mathscr D_0$.

The diagram becomes hence:

\[
\mathscr D_0 \qquad \qquad  \xymatrix{ A \ar@/^1pc/[r] & C
\ar@/^1pc/[l]^e \ar@/^1pc/[r]^e
 & B \ar@/^1pc/[l]^{ef} \ar@(ur,dr)^e }
\]

This allows us to construct an \emph{augmented sequence} recording
both the passages for the sides of the hexagon and of the
parallelogram. To obtain the cutting sequence with respect to $P$
we just need to eliminate the record of the crossing of $A$, $B$
and $C$ by dropping those letters from the cutting sequence and
then substitute $e$ with $A'$ and $f$ with $B'$.

The new sequence in $\{A',B'\}^\Z$ is the cutting sequence of the
same trajectory with respect to the sides of the parallelogram
$P$.

\begin{rk}
We can consider just the case of a trajectory in direction in
$\Sigma_0$ because if the trajectory is in a direction in
$\Sigma_i$ we can apply $\nu_i$ and repeat the procedure. The
unique difference is that we have to consider different diagonals
$e$ and $f$ such that after applying $\nu_i$ they are sent in the
diagonals considered in the case of trajectories in $\Sigma_0$.
\end{rk}

\subsubsection{Relations}

Unfortunately, no matter which derivation procedure do we use, the
derivation and the dictionary do not commute. A simple
counter-example is the trajectory in horizontal direction passing
for the midpoint of the sides $B$ and $C$. Its cutting sequence is
the periodic word obtained repeating the sub-word $BC$. Its
sandwich-derived sequence is the same word and by applying the
dictionary the same trajectory has cutting sequence with respect
to $P$ given by the periodic word obtained repeating $A'B'B'$
whose derivative (both the Series one and the sandwich one) is not
the same sequence.

Another explanation for this is because although the hexagon and
the parallelogram are the same flat tori, by applying the cut and
paste map we are loosing the information of the second marked
point that we have in $S_E$, since after the cut and paste it
becomes a point in the interior of the parallelogram, which gives
a torus with just one marked point corresponding to the four
vertices.

Moreover, if we consider the tessellation of the Teichm\"uller
disk in the case of the hexagon and the one of the square (either
the construction given by Series in \cite{squarehyp} or the one
presented in the previous section) the two seem not to be in any
way comparable. In fact, one sees the disk divided in six sectors
and the second in four sectors, crossing the boundary in different
points.

\section{Cutting sequences for a Bouw-M\"oller surface}\label{fourth}

The last case we will try to analyse in this work is the one of a
Bouw-M\"oller surface.

They were first discovered and described in an algebraic way by
Irene Bouw and Martin M\"oller in \cite{BM} as the surfaces having
Veech group isomorphic to the triangular group
$\Delta(m,n,\infty)$. Hence it is still a Veech surface as the
previous cases of squares, regular $2n$-gons and regular double
odd-gons, which has also triangular groups as Veech groups.

Later, Pat Hooper in \cite{Hooper} gave a flat description of
those surfaces as surfaces obtained by gluing together a finite
number of semi-regular polygons.

Some new results about cutting sequences in Bouw-M\"oller surfaces
are obtained by Diana Davis, who is working on it, but her
approach is deeply different from ours, since she does not use at
all the second representation of Bouw-M\"oller surfaces described
by Hooper, which has a key role in our results. In fact, it
appears in our main result presented in this work, which is
Theorem \ref{main} and state that the derived sequence of a
cutting sequence of a trajectory in the Bouw-M\"oller surface
$\mathscr M(3,4)$ is still a cutting sequence.

\subsection{Two representations of the surface}

As described in \cite{Hooper} to each Bouw-M\"oller surface we can
associate a grid graph and viceversa. Since we have two different
ways of obtaining a decomposition in semi-regular polygons from a
graph, this means that each graph has two
different representations in semi-regular polygons, considering $\mathscr G(m,n)$ or $\mathscr G(n,m)$. Since we always
obtain two-dimensional graphs, each Bouw-M\"oller surface is
identified by two parameters $m$ and $n$ which will be the number
of rows plus one and the number of columns plus one. The corresponding surface
surface will be denoted $\mathscr M(m,n)$. Here we will call $\mathscr M(m,n)$ and $\mathscr M(n,m)$ the two representations of $\mathscr M(m,n)$, because associated to the same graph and affinely equivalent.

Let us start from the representation of the surface $S$ as a
regular octagon with two squares glued to it as in Figure
\ref{BMoctagon}, where all sides have unit length.

\begin{figure}[h]
\centering
\includegraphics[width=1\textwidth]{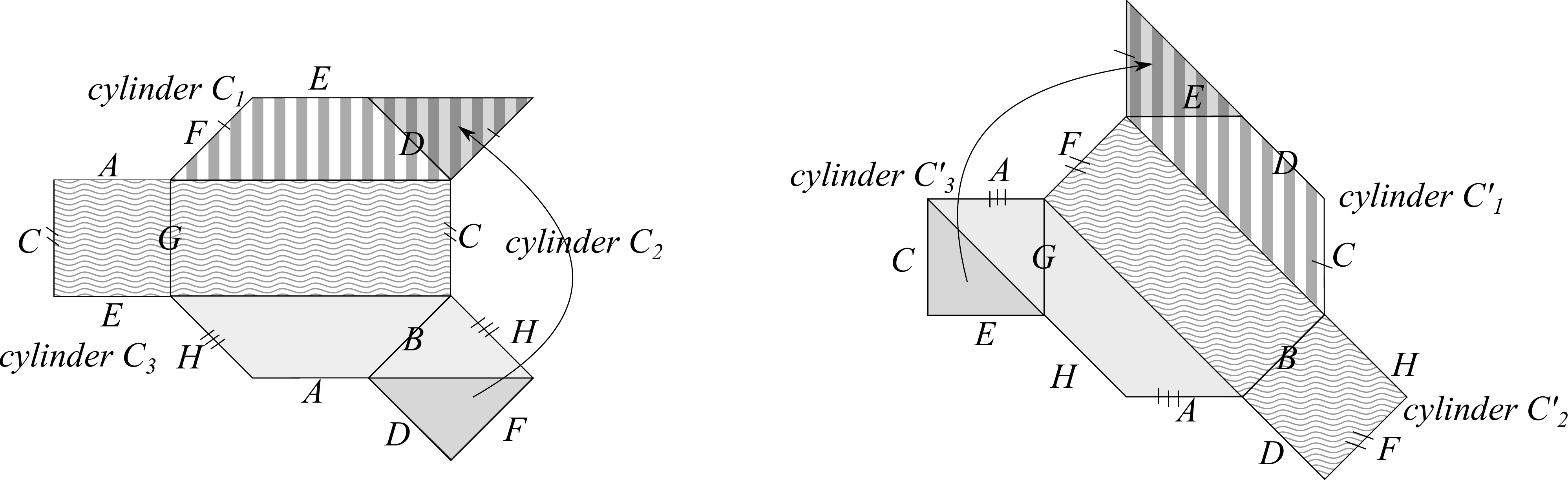}
\caption{First representation of $\mathscr M(3,4)$ and cylinder
decompositions.}
\label{BMoctagon}
\end{figure}

To construct the related grid graph we consider the cylinder
decompositions in Figure \ref{BMoctagon}, the horizontal one and
the one in direction $\frac{3 \pi}{4}$. The set of vertices will
be a bipartite one,  $V=V_b \cup V_w$ and we will denote $V_b$ and
$V_w$ in figures by black dots and white dots. Each vertex
represents a cylinder and vertices in $V_b$ are in 1 to 1
correspondence with horizontal cylinders, while black dots are in
1 to 1 correspondence with cylinders in direction $\frac{3
\pi}{4}$.

An edge between two vertices $v_i$ and $v'_j$ represents the basic
parallelogram which is the intersection of the two corresponding
cylinders $C_i \cap C'_j$.

On each vertex we establish an order on the edges coming out of
that vertex by adding circular arrows turning clockwise on the odd
columns of the graph and counter-clockwise on the even ones and
register which basic parallelograms are glued together.

\begin{figure}[ht]
\centering
\includegraphics[width=0.2\textwidth]{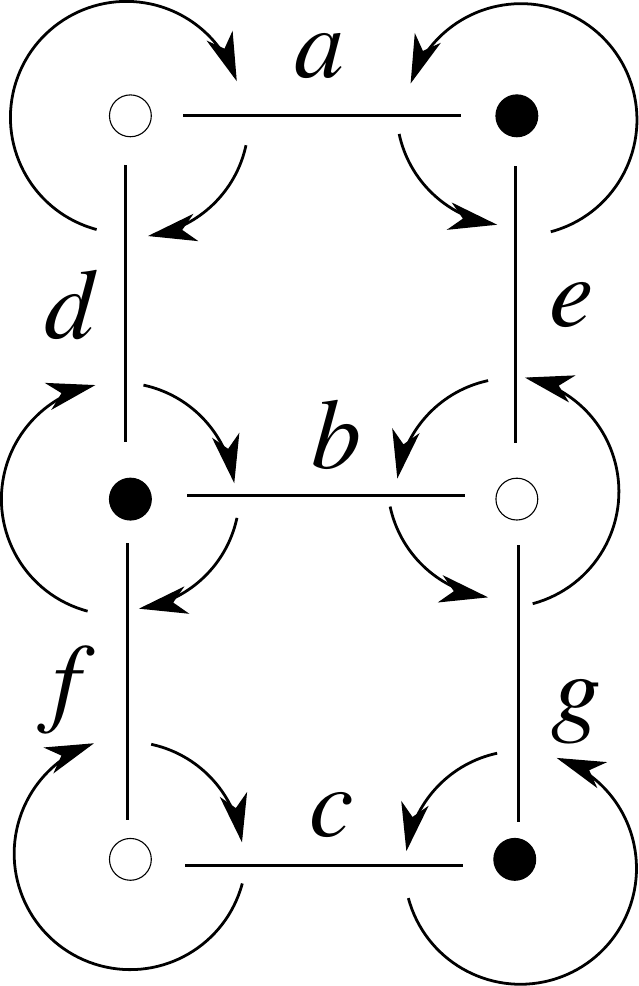}
\caption{Grid graph of $\mathscr M(3,4)$.}
\label{graph}
\end{figure}

In this way we can construct the grid graph that turns out to be
as in Figure \ref{graph}.

As the number or rows and columns are two and three this explains
why we call such a surface $\mathscr M(3,4)$.

Conversely, we can find out the two representations by "decomposing" the diagram in smaller pieces, each of
which corresponds to one polygon (regular or semi-regular) of the
representation.

By decomposing it in vertical, in fact, we obtain the three pieces:

\begin{align*}
\xymatrix{ \circ \ar@{-}[d]^d \\
\bullet \ar@{-}[d]^f \\
\circ} & &  \xymatrix{ \circ \ar@{-}[d]_d \ar@{-}[r]^a & \bullet
\ar@{-}[d]^e \\
\bullet \ar@{-}[d]_f \ar@{-}[r]^b & \circ \ar@{-}[d]^g \\
\circ \ar@{-}[r]_c &\bullet } & & \xymatrix{ \bullet \ar@{-}[d]^e \\
\circ \ar@{-}[d]^g \\
\bullet}
\end{align*}

Now each edge, denoted by a letter, corresponds to a basic
parallelogram (of side length still to be determined) in our
decomposition but in general we do not know the directions of the
cylinder decompositions used and hence the slope of the
parallelograms. Up to shear the representation obtained, we hence
now consider them as basic rectangles. The ones repeated twice
correspond each time to half a rectangle (i.e. a triangle) in
each piece of the decomposition.

The arrows between edges that we have on the graph show us how the
basic rectangles are glued. Since vertices in $V_w$ represent
horizontal cylinders, an arrow rotating around a vertex in $V_w$
between two edges means that the right side of the basic rectangle
corresponding to one edge is glued to the left side of the basic
rectangle corresponding to the other edge. Conversely, since
vertices in $V_b$ represent vertical cylinders, arrows around
vertices in $V_b$ correspond to basic rectangles glued through top
and bottom sides.

For example let us compare the second piece of the graph and the
second polygon in Figure \ref{octort} which shows the first
representation. In the graph, there is an arrow connecting edges
$a$ and $d$ around a white vertex and hence $a$ and $d$ are glued
on the right and left side respectively. The arrow between edges
$d$ and $b$ is around a black vertex and hence they are glued
through the top and bottom sides. Moreover, the basic rectangle
corresponding to $b$ is glued to $d$ and $f$ by the top and bottom
sides since the arrows connecting the edge $b$ and the edges $d$
and $f$ are around a black vertex. And it is glued to $e$ and $g$
by the right and left sides because the arrows connecting them are
around a white vertex.

This allows us to give the first decomposition $R_1^\perp$ (see
Figure \ref{octort}) that turns out to be exactly as before, two
quadrilaterals and a octagon, but sheared so that the two cylinder
decomposition described before are orthogonal. To have back the
regular decomposition we have to apply the shear sending the
orthogonal lines back in direction $\frac{3 \pi}{4}$. We will call
such sheared representation $R_1$ and the one with orthogonal
cylinder decompositions $R_1^\perp$.

\begin{figure}[h]
\centering
\includegraphics[width=0.9\textwidth]{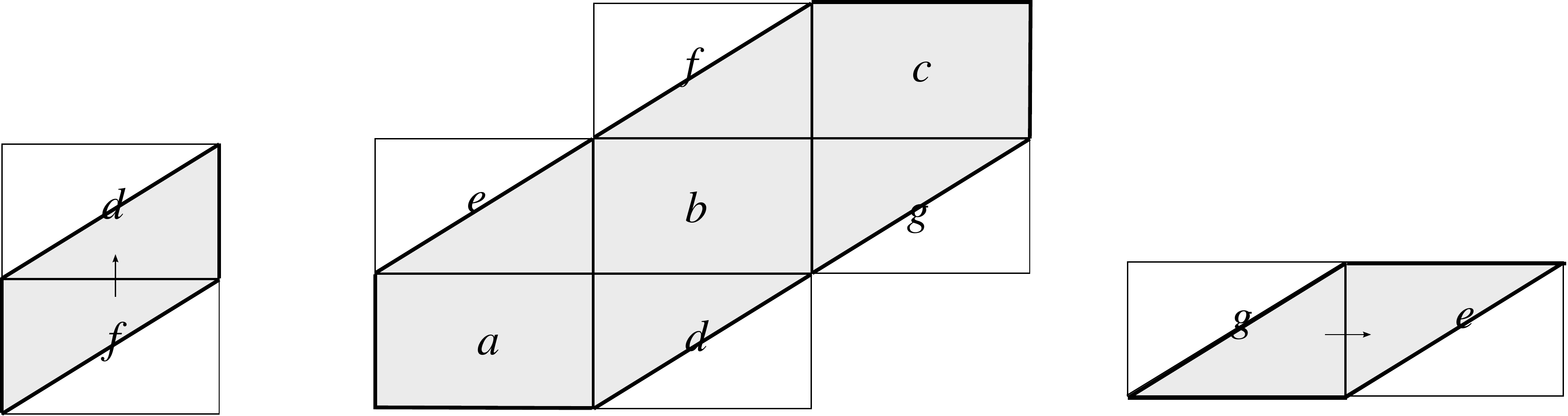}
\caption{The representation $R_1^\perp$.} \label{octort}
\end{figure}

By decomposing the grid graph in horizontal, we obtain the second
representation. We get the decomposition

\begin{align*}
\xymatrix{ \circ \ar@{-}[r]^a & \bullet} & & \xymatrix{ \circ
\ar@{-}[d]_d \ar@{-}[r]^a & \bullet \ar@{-}[d]^e \\
\bullet \ar@{-}[r]_b & \circ} & & \xymatrix{\bullet
\ar@{-}[d]_f \ar@{-}[r]^b & \circ \ar@{-}[d]^g \\
\circ \ar@{-}[r]_c & \bullet} & & \xymatrix{ \circ \ar@{-}[r]^c &
\bullet}
\end{align*}

and repeating the same procedure described above we have the
decomposition in Figure \ref{hexort} that we will denote
$R_2^\perp$. It can be sheared back in two triangles and two
semi-regular hexagons. We will call the second representation
$R_2$ and is represented in Figure \ref{BMhexagon} together with
the two cylinder decompositions which becomes orthogonal in the
representation $R_2^\perp$.

\begin{figure}[h]
\centering
\includegraphics[width=0.9\textwidth]{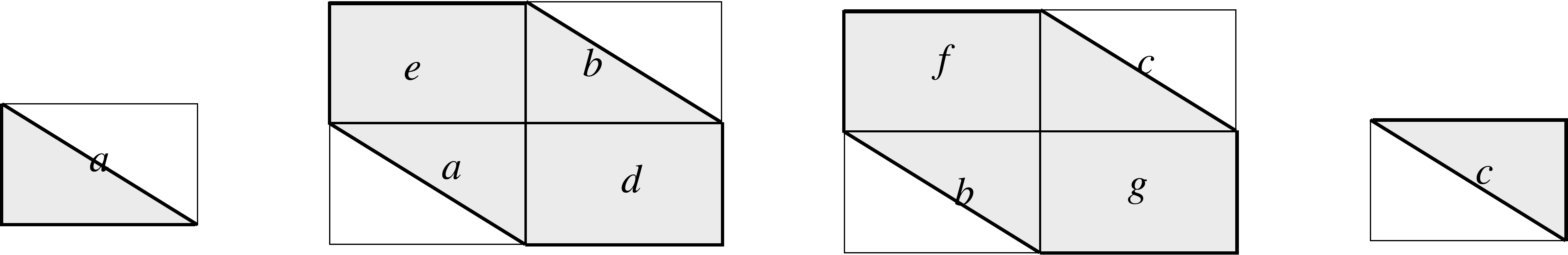}
\caption{The representation of $R_2^\perp$.} \label{hexort}
\end{figure}

In \cite{Hooper}, Hooper describes how to determine the lengths of
the sides of the basic rectangles and hence the sides of the
polygons from the graph. The first decomposition is made of
regular polygons, a kind of particular case of semi-regular
polygons where one of the sizes of sides is zero. So all sides has
same length that we fix to be 1. In the case of the hexagon, we
find out that the two triangles are equilateral and with sides of
the same length of the short sides of the hexagons and
semi-regular hexagons has sides of length 1 and $\frac{\sqrt
2}{2}$. Naturally, this is not necessarily the representation
corresponding to the same surface of $R_1$, but keeping the ration
of the two sides we can proportionally change the sizes of the
sides to have the same surface. To find the new sizes we impose
that the total area of the polygons appearing in $R_2$ must be the
same of the total area of the polygons in $R_1$ provided it has
all sides of length 1. A simple calculus shows that the area of
the polygons in the first representation is
\[
\mathcal A_1=2(2+ \sqrt 2)
\]
and denoting by $a$ and $a'=\frac{\sqrt 2}{2}a$ the two side length
in semi-regular hexagons the area of the second representation is
\[
\mathcal A_2=\sqrt 3(1+\sqrt 2)a.
\]

We hence impose that $\mathcal A_1=\mathcal A_2$ and obtain that
the length of the sides in the semi-regular hexagon must be
\[
a=\sqrt{\frac{2\sqrt 6}{3}} \qquad \text{and} \qquad
a'=\frac{\sqrt 2}{2}a= \frac{\sqrt 2}{2}\sqrt{\frac{2\sqrt 6}{3}}.
\]

\begin{figure}[h]
\centering
\includegraphics[width=1\textwidth]{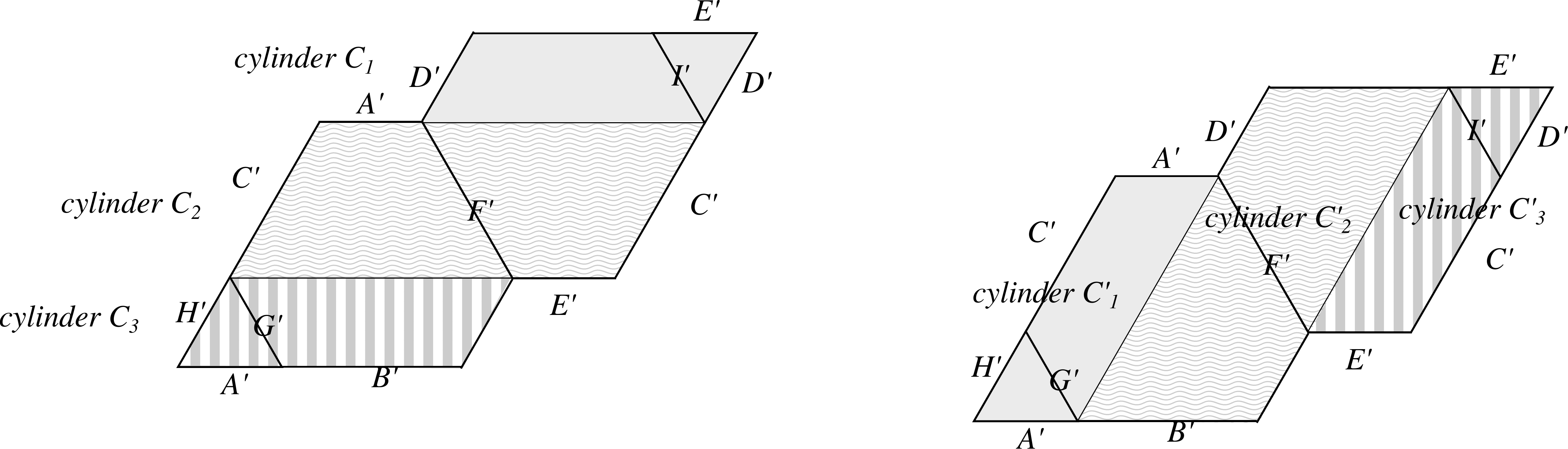}
\caption{Second representation of $\mathscr M(3,4)$.}
\label{BMhexagon}
\end{figure}

\subsection{Transition diagrams}

Let us now consider a linear trajectory in the surface $R_1$. Again, we can suppose that the
trajectory does not hit any vertex and code such trajectory by
recording the sides crossed by it considering the labelling given
in Figure \ref{BMoctagon}, obtaining a bi-infinite word in a new
alphabet $\mathscr A =\{A,B,C,D,E,F,G,H\}$.

Again, we try to study the cutting sequence of the trajectory.
Naturally, being each side of the squares glued to one in the
octagon, it does not matters where we collocate the two squares
with respect to the octagon as long as we preserve the glueing.

Since in this case the rotation of angle $\pi$ does not preserve
the labels, we have to consider trajectories in all directions in
$[0,2\pi]$. Sectors of total angle $\frac{\pi}{8}$ are the most
natural for having the transitions in the cutting sequences
determined and hence we divide the directions in 16 sectors of
size $\frac{\pi}{8}$. For reasons that will become clear later on,
we will call $\Sigma_0$ the sector $[\frac{7 \pi}{8},\pi]$ and
enumerate the others clockwise so that $\Sigma_i=[\pi -\frac{(i+1)
\pi}{8},\pi -\frac{i \pi}{8}]$ for $i=0,\dots,15$. We have
transformations of $R_1$ sending each $\Sigma_i$ back in
$\Sigma_0$ and preserving the figure and they are given by:
\begin{align*}
\nu_0&=\id, & \nu_1&=\rho_{-\frac{\pi}{8}} \circ r_1 \circ
\rho_{\frac{\pi}{8}}, & \nu_2&=\rho_{\frac{\pi}{4}}, &
\nu_3&=\rho_{-\frac{3 \pi}{8}} \circ
r_1 \circ\rho_{\frac{3 \pi}{8}}, \\
\nu_4&=\rho_{\frac{\pi}{2}}, & \nu_5&=\rho_{\frac{\pi}{8}} \circ
r_2 \circ \rho_{-\frac{\pi}{8}}, & \nu_6&=\rho_{\frac{3 \pi}{4}},
& \nu_7&=r_2, \\
\nu_8&=\nu_7 \circ r_1, & \nu_9&=\nu_6 \circ r_1, &
\nu_{10}&=\nu_5 \circ r_1, & \nu_{11}&=\nu_4 \circ r_1, \\
\nu_{12}&= \nu_3 \circ r_1, & \nu_{13}&= \nu_2 \circ r_1, &
\nu_{14}&= \nu_1 \circ r_1, & \nu_{15}&= r_1.
\end{align*}

We recall that $\rho_\theta$ is the counter-clockwise rotation of
angle $\theta$ and $r_1$ and $r_2$ are respectively the reflection
with respect to the horizontal and vertical axis.

We computed the corresponding permutations and hence proved the
following:

\begin{lemma}
The permutations on the labels of $R_1$ corresponding by the
action of the matrices above are
\begin{align*}
\pi_0&=\id, & \pi_1&=(AD)(BC)(EH)(FG), \\
\pi_2&= (ABCDEFGH), & \pi_3&=(AC)(DH)(EG), \\
\pi_4&=(ACEG)(BDFH), & \pi_5&=(AB)(CH)(DG)(EF), \\
\pi_6&=(ADGBEHCF), & \pi_7&= (BH)(CG)(DF), \\
\pi_8&=(AE)(BF)(CG)(DH), & \pi_9&=(AH)(BG)(CF)(DE), \\
\pi_{10}&=(AFCHEBGD), & \pi_{11}&=(AG)(BF)(CE), \\
\pi_{12}&=(AGEC)(BHFD), & \pi_{13}&=(AF)(BE)(CD)(GH), \\
\pi_{14}&=(AHGFEDCB), & \pi_{15}&=(AE)(BD)(FH).
\end{align*}
\end{lemma}

We now investigate which transitions may occur in a cutting
sequence by writing down the transition diagram corresponding to a
trajectory in $\Sigma_0$. The transition diagrams for trajectories
in other directions would follow directly from this one by
applying the corresponding permutations. The first transition
diagram is:

\[
\mathscr D_0 \qquad \qquad  \xymatrix{ \quad H \quad \ar@/^1pc/[r]
& \quad B \quad \ar@/^1pc/[l] \ar@/^1pc/[d]
& \quad F \quad \ar@/^1pc/[l] \ar@/^1pc/[r] & \quad D \quad \ar@/^1pc/[l] \ar@/^1pc/[d] \\
\quad A \quad \ar@/^1pc/[u] & \quad G \quad \ar@/^1pc/[l]
\ar@/^1pc/[r] &  \quad C \quad \ar@/^1pc/[l] \ar@/^1pc/[u] & \quad
E \quad \ar@/^1pc/[l]}
\]

\bigskip
We note that a cutting sequence in direction in $\Sigma_i$
is a bi-infinite path in the corresponding transition diagram
$\mathscr D_i$.

\begin{rk}
Given a trajectory $\tau$ in direction in $\Sigma_i$ its normal
form is the trajectory in direction in $\Sigma_0$
\[
n(\tau)=\nu_k \tau.
\]

Moreover, cutting sequences verify
\[
c(n(\tau))=\pi_k \cdot c(\tau).
\]
\end{rk}

\subsection{Derivation}

In this section we will prove a statement equivalent to Proposition \ref{central} for the hexagon. We need hence to define
some kind of combinatorial operation on words (the derivation) so
that the derived sequence of a cutting sequence will be again a
cutting sequence of an other trajectory. Here we
explain the role of the second representation: after deriving, we
will effectively have the cutting sequence of another trajectory, but coded with respect to the sides of the
second representation. This means that the derivation will be
nothing else but the rule of how to pass from a cutting sequence
in $R_1$ and the cutting sequence in $R_2$ after applying the
matrix that sends $R_1$ to $R_2$. To conclude the procedure and get a cutting sequence of a new trajectory in the same surface, we have to do the same procedure on $R_2$.

As before, we have:

\begin{defin}
A word in the alphabet $w \in \alfab$ is admissible if it
represents a bi-infinite path in one of the diagrams $\mathscr
D_i$.
\end{defin}

and by construction

\begin{lemma}
If $w \in \alfab$ is a cutting sequence, it is admissible in one
of the diagrams $\mathscr D_i$.
\end{lemma}

Let us first define the derivation for sequences admissible in
diagram $\mathscr D_0$.

\begin{defin}\label{derivation}
Let $w$ be a word in $\alfab$ admissible in diagram $\mathscr
D_0$. Its derived sequence is the sequence in the new alphabet
$\mathscr A'=\{A',B',C',D',E',F',G',H',I'\}$ obtained from $w$ by
interpolating it with the letters on the arrows of the following
diagram, changing $A,D,E,H$ with the correspondent primed and
dropping letters $B,C,F,G$.

\begin{align*}
\mathscr D_0 \qquad \qquad  \xymatrix{ \quad H \quad \ar@/^1pc/[r]
& \quad B \quad \ar@/^1pc/[l]^{G'} \ar@/^1pc/[d]
& \quad F \quad \ar@/^1pc/[l]_{B'} \ar@/^1pc/[r] & \quad D \quad \ar@/^1pc/[l]^{I'} \ar@/^1pc/[d] \\
\quad A \quad \ar@/^1pc/[u] & \quad G \quad \ar@/^1pc/[l]
\ar@/^1pc/[r]_{C'} &  \quad C \quad \ar@/^1pc/[l]^{F'}
\ar@/^1pc/[u] & \quad E \quad \ar@/^1pc/[l]}
\end{align*}
\end{defin}

In a similar way we can define derivation on words admissible in
the other diagrams.

\subsection{Veech element}

The main result that we have proved is the following:

\begin{theo}\label{main}
Let us consider a trajectory $\tau$ in $R_1$ in direction $\theta
\in \Sigma_0$ with cutting sequence $c(\tau)$. The derived
sequence of the cutting sequence $c(\tau)$ is still a cutting
sequence.
\end{theo}

We will now prove the theorem so that it describes how to pass
from cutting sequences in one representation to cutting sequences
in the other. We will hence find the deformations we need to apply
to pass from one representation to the other.

\begin{figure}[h]
\centering
\includegraphics[width=0.9\textwidth]{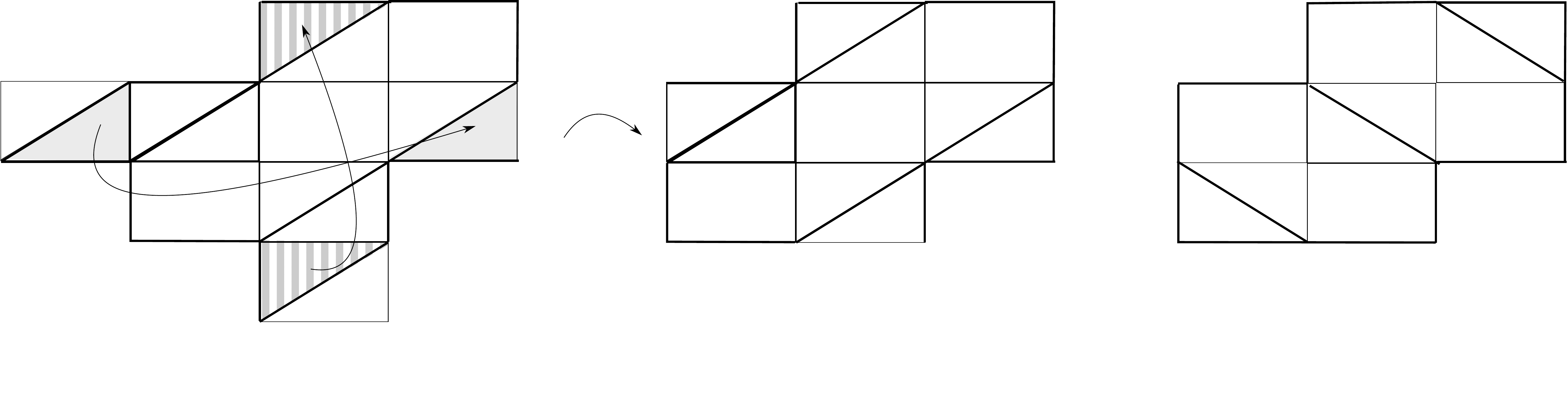}
\caption{The representations $R_1^\perp$ and $R_2^\perp$.} \label{ort}
\end{figure}

\begin{proof}

\emph{First step.} The key observation for finding out how to go
from one representation to the other is made by comparing the
polygonal decompositions $R_1^\perp$ and $R_2^\perp$, which were
obtained from the graph as explained. As one can easily see
looking at Figure \ref{ort}, $R_1^\perp$ and $R_2^\perp$ are the
same up to cutting and pasting two half rectangles, translation
equivalence that we will call $\tilde \Upsilon_S$. We call its
image $\tilde R_1^\perp$.

\begin{figure}[h]
\centering
\includegraphics[width=0.5\textwidth]{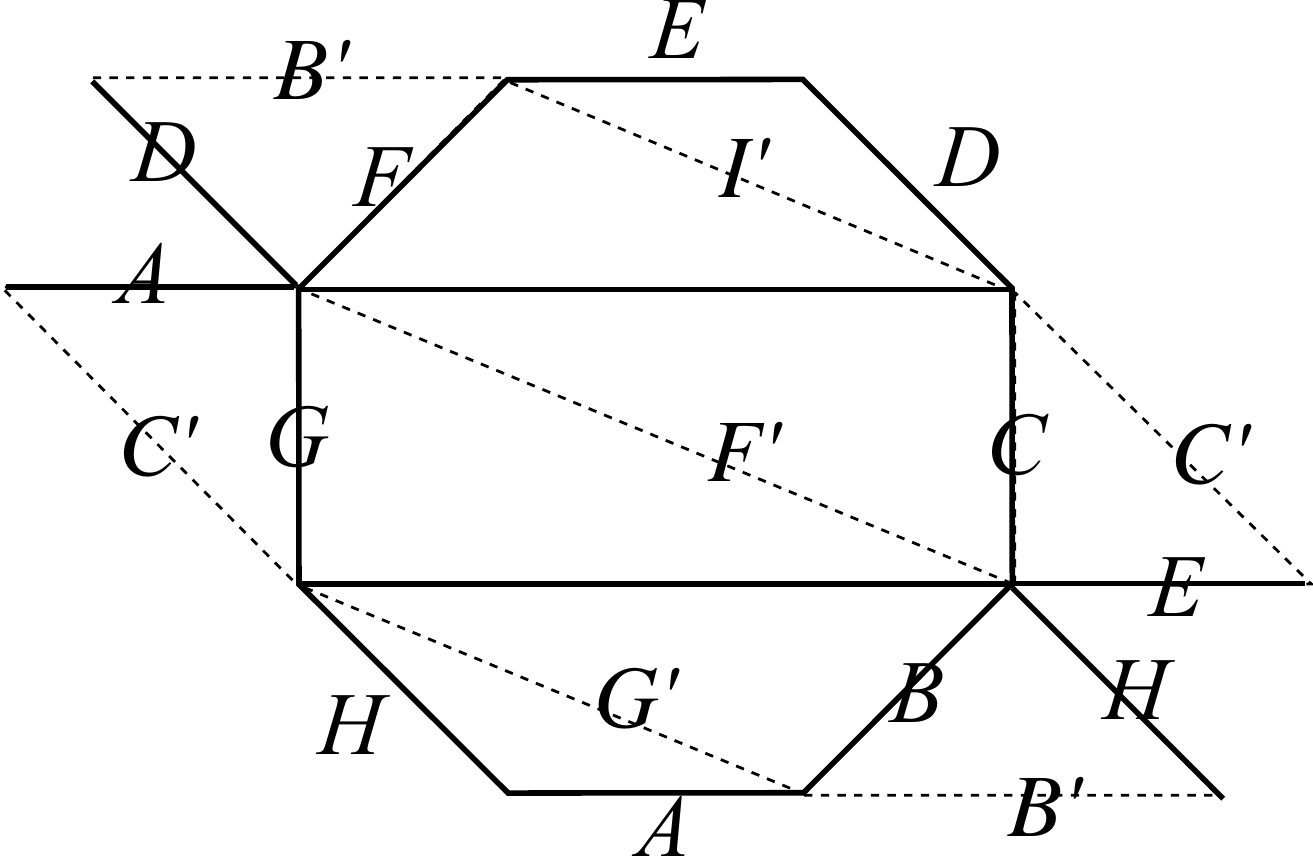}
\caption{$\Upsilon_S R_1=\tilde R_1$.} \label{oct}
\end{figure}

Naturally, we can apply the induced cut and paste to the
representation $R_1$ and we will call this map $\Upsilon_S$,
obtaining a representation $\tilde R_1$. We remark that this map
has in fact the same role as $\Upsilon_{S_E}$ in the previous
section. Moreover, this figure allows us to see how $R_2$ can be
obtained by shearing $R_1$ after applying $\Upsilon_S$.

\emph{Second step.}  On cutting sequences, $\Upsilon$ act by
adding some necessary sides that we will call auxiliary sides from
now on (see Figure \ref{oct}) and dropping some others. We can
hence record in the diagram of admissible transitions the crossing
of auxiliary sides. It is clear that some sides are shared from
the two configurations, precisely $A,D,E,H$, some must be added
and are the one on the augmented diagram (\ref{derivation}) and
some more must be eliminated, because they are not part of the
sides of the new configuration, precisely $B,C,F,G$. This step is
equivalent to the third step of the proof of the similar
proposition for the hexagon, when we applied the cut and paste to
the new hexagon and wrote the cutting sequence $\overline c(\tau)$
with respect to the new sides. It turned out to be, in fact,
exactly the derived sequence of the original cutting sequence.

\emph{Third step.} In the previous case the next step was to send
the sheared hexagon back in the previous one, by applying a Veech
element. In this case, we send the
configuration $\tilde R_1$ to $R_2$ instead.

To send $\tilde R_1$ to $R_2$, we apply three different
transformations and consider their composition $\sigma=\sigma_3
\circ \sigma_2 \circ \sigma_1$.

\begin{itemize}
\item The first one, $\sigma_1$ will bring $\tilde R_1$ in $R_2^\perp$,
easier to compare with $R_2^\perp$ and hence with $R_2$.

Since the two decompositions were in directions $0$ and $\frac{3
\pi}{4}$ we choose a shear to make them orthogonal, given by the
matrix
\[
\sigma_1=\begin{pmatrix} 1 & 1 \\ 0 & 1 \end{pmatrix}.
\]

\item The second one in to adjust lengths and obtain from $R_1^\perp$
the configuration $R_2^\perp$. We want it to act as in Figure
\ref{length} and hence we use the matrix:
\[
\sigma_2=\begin{pmatrix} \frac{\sqrt 2}{2}\sqrt{\frac{2 \sqrt
6}{3}} & 0 \\ 0 & \frac{\sqrt 3}{2}\sqrt{\frac{2 \sqrt 6}{3}}
\end{pmatrix}.
\]

\begin{figure}[ht]
\centering
\includegraphics[width=0.7\textwidth]{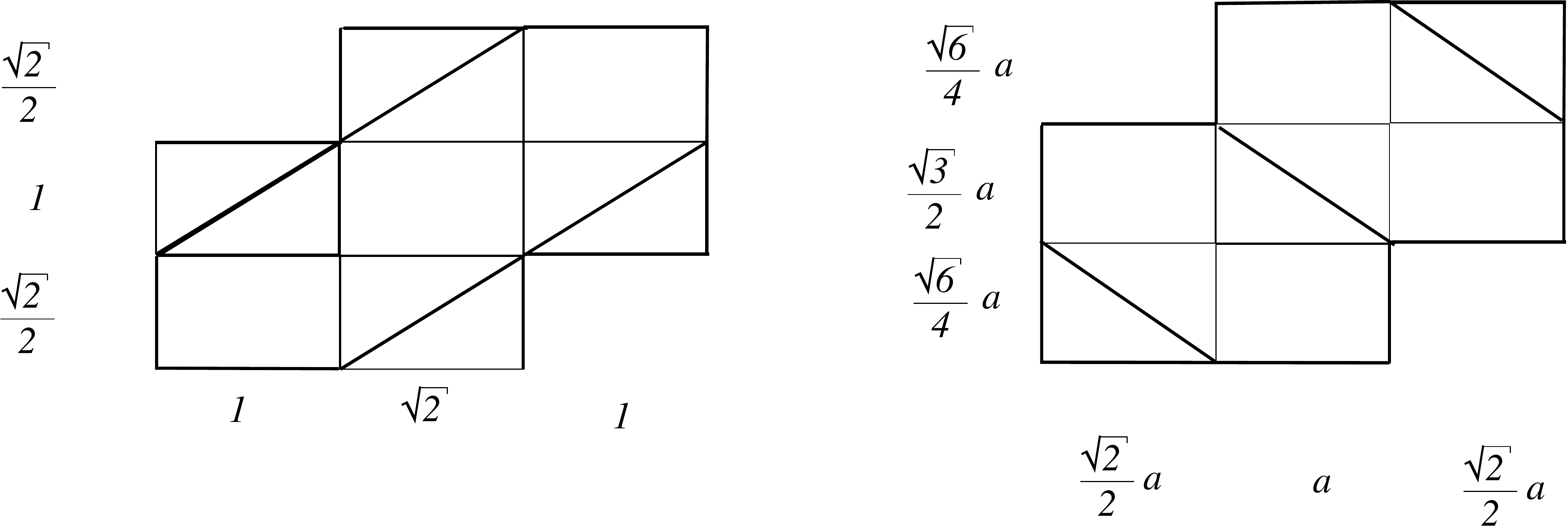}
\caption{The lengths comparison between $R_1^\perp$ and
$R_2^\perp$.} \label{length}
\end{figure}

\item The third one will send $R_2^\perp$ back in the original configuration
$R_2$. It hence sends vertical directions in the directions of the
cylinder decompositions in $R_2$ corresponding to the ones in
$R_1$. Those decompositions are the ones in direction $0$ and the
one in direction $\frac{\pi}{3}$ as shown in Figure
\ref{BMhexagon}. The shear that sends vertical lines in lines in
direction $\frac{\pi}{3}$ is:
\[
\sigma_3=\begin{pmatrix} 1 & \frac{\sqrt 3}{3} \\ 0 & 1
\end{pmatrix}.
\]

\end{itemize}

Finally, we consider the composition of the three of them
\[
\xymatrix{ \tilde R_1 \ar@/_1pc/[rrr]_\gamma \ar[r]^{\sigma_1} &
R_1^\perp \ar[r]^{\sigma_2} & R_2^\perp \ar[r]^{\sigma_3} & R_2}
\]

which is given by the matrix

\[
\sigma= \sigma_3 \circ \sigma_2 \circ \sigma_1= \begin{pmatrix}
\frac{\sqrt 2}{2}\sqrt{\frac{2 \sqrt 6}{3}} &
\frac{1}{2}\sqrt{\frac{2 \sqrt 6}{3}}(1+\sqrt 2) \\
0 & \frac{\sqrt 3}{2}\sqrt{\frac{2 \sqrt 6}{3}}
\end{pmatrix}.
\]

Geometrically, the action in showed in Figure \ref{final}.

\begin{figure}[h]
\centering
\includegraphics[width=0.7\textwidth]{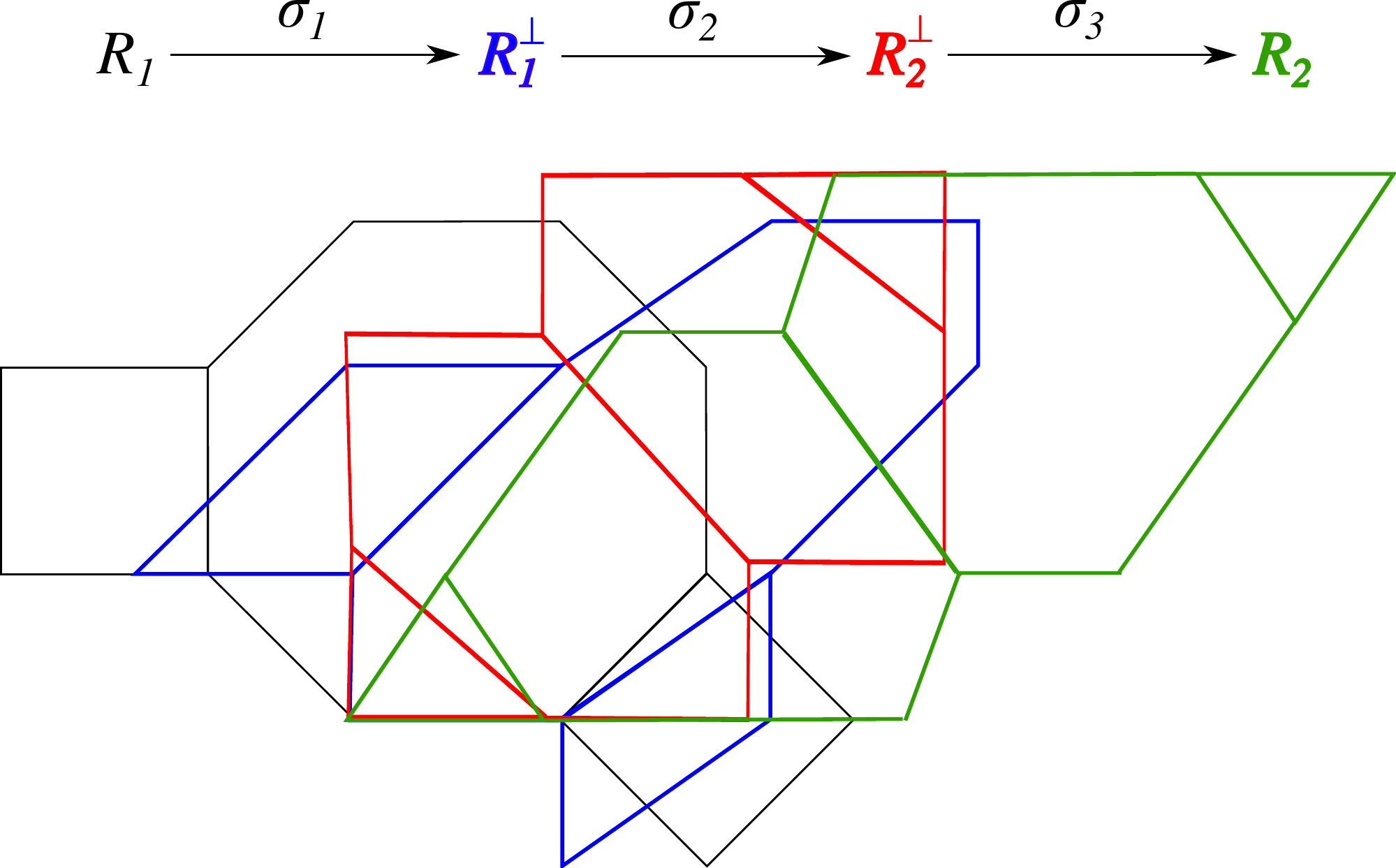}
\caption{The action of $\sigma$.}
\label{final}
\end{figure}

\emph{Fourth step} Since $R_1$ and $R_2$ represent the same
surface this must be in the Veech group of it. As explained for
the hexagon, we can consider the affine equivalence $\Psi_\sigma
\colon S \to S$ obtained by the composition of $\Upsilon \colon S \to S$
and the deformation $\Phi_\sigma \colon S \to S$. $\Psi_\sigma$ is
the affine equivalence induced by the Veech element.

\end{proof}

We actually proved something more precise than Theorem \ref{main}
that is the following proposition, similar to the one in Section \ref{central}:

\begin{prop}
Given a trajectory $\tau$ in $R_1$ in direction $\theta \in \Sigma_0$ with
cutting sequence $c(\tau)$, let us consider the trajectory $\tau'=
\Psi_\sigma \tau$ in $R_2$. We have that the derived sequence of the
original trajectory is the cutting sequence of the new trajectory.
This means
\[
c(\tau')=c(\tau)'.
\]
\end{prop}

The proof follows immediately from the last step of the proof of
the theorem.

\subsection{Future works}

To obtain a complete characterization or at least a complete
necessary condition, we need to do something more. The next step
will be in fact to do the same thing for the configuration $R_2$,
finding out how to shear, scale and shear it again to obtain a
sheared $R_1$ that can be sent back in the original configuration
so that the procedure can start again and we can affirm that a
cutting sequence is infinitely derivable, defining the derivation
as alternating the one for $R_1$ and the one we will give for
$R_2$.

The idea we are following is to consider the symmetric cylinder
decomposition in direction $\frac{2 \pi}{3}$ in $R_2$ together
with the horizontal one and apply the same procedure. The first thing is to shear $R_2$ so that the two cylinder decompositions become orthogonal. Then we are
planning to modify lengths again to return to the configuration
where the cylinder decompositions symmetric to the ones we
considered until now in $R_1$ (which means, the horizontal one and the one
in direction $\frac{\pi}{4}$) are orthogonal. Finally, we want to
shear it again to a sheared representation $R_1$.

Applying the correspondent matrices of the action we just
explained gives a shear:
\begin{multline*}
\gamma= \begin{pmatrix} 1 & 1 \\ 0 & 1 \end{pmatrix}
\begin{pmatrix} \sqrt 2 \sqrt{\frac{2 \sqrt 6}{3}} & 0 \\ 0 &
\frac{2 \sqrt 3}{3} \sqrt{\frac{2 \sqrt 6}{3}} \end{pmatrix}
\begin{pmatrix} 1 & \frac{\sqrt 3}{3} \\ 0 & 1 \end{pmatrix} \cdot \\
\cdot \begin{pmatrix} 1 & \frac{\sqrt 3}{3} \\ 0 & 1 \end{pmatrix}
\begin{pmatrix} \frac{\sqrt 2}{2} \sqrt{\frac{2 \sqrt 6}{3}} & 0 \\ 0 &
\frac{\sqrt 3}{2} \sqrt{\frac{2 \sqrt 6}{3}} \end{pmatrix}
\begin{pmatrix} 1 & 1 \\ 0 & 1 \end{pmatrix}
=\begin{pmatrix} 1 & 2+\sqrt 2 \\ 0 & 1 \end{pmatrix}
\end{multline*}
which has as argument exactly the modulus of all the three
cylinders in the horizontal decomposition. This is in fact the
shear that Diana Davis is using in her works about Bouw-M\"oller
surfaces. However, her method does not consider the two
representations of Bouw-M\"oller surfaces and try to pass directly
again to the same representation. ~In this way, she cannot cover all sectors of directions of trajectories. Even if our method seems to
result a bit longer because we have a intermediate step, each step
turns out to be much easier passing for the second
representations, more similar to the one used in \cite{octagon}
and more useful to understand what is happening in the
Teichm\"uller disk of deformations.

The next step will be to find the position of the sides of the
sheared $R_2$ and record them in the diagram of admissible
transitions.

The reason for choosing the numeration of sectors as we did is
that in this way, as we had in the previous cases, the natural
range $\Sigma_0$ of directions for $R_1$ is send in the new
configuration in a bigger interval of directions: it is sent in
$[\frac{2 \pi}{3}, \pi]$ which as range double of the natural
angle for this configuration, which is $\frac{\pi}{6}$, to have
the transition diagram. We will hence need some way of normalizing
the trajectory and send them back in the new $\Sigma'_0$. In fact,
for $R_2$ we will divide the range of possible directions
$[0,2\pi]$ in 12 intervals $\Sigma'_i=[\pi -\frac{(i+1)
\pi}{6},\pi -\frac{i \pi}{6}]$ for $i=0, \dots, 11$ and
trajectories in $\Sigma_0$ will be sent in trajectories either in
$\Sigma_0$ or $\Sigma_1$.

Other future plans are to write down the two equivalent of the
Farey maps and describe how to apply the two alternating between
them as we alternate configurations.

Moreover, it would be meaningful to try to find out what is
happening in the Teichm\"uller disk while applying the matrices
described before (naturally, all of them are in $SL(2,\R)$ because
they need to preserve areas). We would expect to have a tree
equivalent to the one for the other case, but bipartite, with
vertices with $m$ edges and vertices with $n$ edges corresponding
to the different configurations.

Finally, we will try to understand how to generalize this method
to a general Bouw-M\"oller surface.

\clearpage
\addcontentsline{toc}{section}{\refname}
\bibliographystyle{alpha}
\bibliography{biblio}

\begin{thebibliography}{Ser85b}

\bibitem[BM10]{BM}
Irene~I. Bouw and Martin M{\"o}ller.
\newblock Teichm\"uller curves, triangle groups, and {L}yapunov exponents.
\newblock {\em Annals of Mathematics (2)}, 172(1):139--185, 2010.

\bibitem[Dav13]{pentagon}
Diana Davis.
\newblock Cutting sequences, regular polygons, and the {V}eech group.
\newblock {\em Geometriae Dedicata}, 162:231--261, 2013.

\bibitem[Dav14]{BMDavis}
Diana Davis.
\newblock Cutting sequences on bouw-m\"oller surfaces.
\newblock In preparation, 2014.

\bibitem[Hoo08]{H}
Patrick~W. Hooper.
\newblock The bouw-m\"oller lattice surfaces and eigenvectors of grid graphs.
\newblock 2008.

\bibitem[Hoo12]{Hooper}
Patrick~W. Hooper.
\newblock Grid graphs on lattice surfaces.
\newblock To appear in International Mathematics Research Notices, 2012.

\bibitem[HS06]{veechhub}
Pascal Hubert and Thomas~A. Schmidt.
\newblock An introduction to {V}eech surfaces.
\newblock In {\em Handbook of dynamical systems. {V}ol. 1{B}}, pages 501--526.
  Elsevier B. V., Amsterdam, 2006.

\bibitem[Mas06]{transurf}
Howard Masur.
\newblock Ergodic theory of translation surfaces.
\newblock In {\em Handbook of dynamical systems. {V}ol. 1{B}}, pages 527--547.
  Elsevier B. V., Amsterdam, 2006.

\bibitem[Ser82]{squareintro}
Caroline Series.
\newblock Non-{E}uclidean geometry, continued fractions, and ergodic theory.
\newblock {\em The Mathematical Intelligencer}, 4(1):24--31, 1982.

\bibitem[Ser85a]{square}
Caroline Series.
\newblock The geometry of {M}arkoff numbers.
\newblock {\em The Mathematical Intelligencer}, 7(3):20--29, 1985.

\bibitem[Ser85b]{squarehyp}
Caroline Series.
\newblock The modular surface and continued fractions.
\newblock {\em Journal of the London Mathematical Society (2)}, 31(1):69--80,
  1985.

\bibitem[SU10]{octagonteich}
John Smillie and Corinna Ulcigrai.
\newblock Geodesic flow on the {T}eichm\"uller disk of the regular octagon,
  cutting sequences and octagon continued fractions maps.
\newblock In {\em Dynamical numbers---interplay between dynamical systems and
  number theory}, volume 532 of {\em Contemp. Math.}, pages 29--65. Amer. Math.
  Soc., Providence, RI, 2010.

\bibitem[SU11]{octagon}
John Smillie and Corinna Ulcigrai.
\newblock Beyond {S}turmian sequences: coding linear trajectories in the
  regular octagon.
\newblock {\em Proceedings of the London Mathematical Society. Third Series},
  102(2):291--340, 2011.

\bibitem[Vee89]{veechlattice}
W.~A. Veech.
\newblock Teichm\"uller curves in moduli space, {E}isenstein series and an
  application to triangular billiards.
\newblock {\em Inventiones Mathematicae}, 97(3):553--583, 1989.

\end{thebibliography}
\nocite{*}

%\clearpage

%\addcontentsline{toc}{section}{\refname}
%\printbibliography
%\nocite{*}

\end{otherlanguage}
\end{document}